\documentclass[12pt]{amsart}

\usepackage{amsmath}
\usepackage{amssymb,mathtools,amsthm,latexsym,bm}

\calclayout

\title{The volume of the boundary of a Sobolev $(p, q)$-extension domain}
\author{Pekka Koskela, Alexander Ukhlov
and Zheng Zhu}

\address{Pekka Koskela\\
Department of Mathematics and Statistics\\
University of Jyv\"askyl\"a, P.O. Box 35 (MaD),
FI-40014, Jyv\"askyl\"a, Finland}
\email{\tt pekka.j.koskela@jyu.fi}

\address{Alexander Ukhlov\\
Department of Mathematics\\
Ben Gurion University of the Negev, P.O.B. 653, Be'er Sheva 84105, Israel}
\email{\tt ukhlov@math.bgu.ac.il}

\address{Zheng Zhu\\
Department of Mathematics and Statistics\\
University of Jy\"askyl\"a, P.O. Box 35 (MaD),
FI-40014, Jyv\"askyl\"a, Finland}
\email{\tt zheng.z.zhu@jyu.fi}

\usepackage{stackrel}

\usepackage{color}

\numberwithin{equation}{section}
\long\def\colred#1\endred{{\color{red}#1}}
\long\def\colgreen#1\endgreen{{\color{green}#1}}
\long\def\colmagenta#1\endmagenta{{\color{magenta}#1}}
\long\def\colblue#1\endblue{{\color{blue}#1}}
\long\def\colyellow#1\endyellow{{\color{yellow}#1}}

\usepackage{tikz}
\usetikzlibrary{matrix,arrows}

\theoremstyle{plain}
\newtheorem{thm}{Theorem}[section]
\newtheorem{lem}{Lemma}[section]
\newtheorem{prop}{Proposition}[section]
\newtheorem{rem}{Remark}[section]
\newtheorem{cor}{Corollary}[section]
\newtheorem{defn}{Definition}[section]

\newtheorem{quest}{Question}[section]

\numberwithin{equation}{section}

\theoremstyle{remark}
\newtheorem{remark}[equation]{Remark}

\theoremstyle{definition}

\newtheorem*{question*}{Question}

\usepackage{graphicx}


\thanks{The first and third authors have been supported  by the Academy of Finland (project No. 323960). The third author thanks  Tero Kilpel\"ainen for introducing him to fine topology. }

\newcounter{prob}
\def\rr{{\mathbb R}}
\def\rn{{{\rr}^n}}

\def\fz{\infty}

\def\dist{{\mathop\mathrm{\,dist\,}}}

\def\boz{{\Omega}}

\def\bint{{\ifinner\rlap{\bf\kern.25em--}
\int\else\rlap{\bf\kern.45em--}\int\fi}\ignorespaces}

\def\bbint{{\ifinner\rlap{\bf\kern.25em--}
\hspace{0.078cm}\int\else\rlap{\bf\kern.45em--}\int\fi}\ignorespaces}

\def\sC{\mathsf C}

\def\diam{{\mathop\mathrm{\,diam\,}}}

\def\r{\right}
\def\lf{\left}

\def\XXint#1#2#3{{\setbox0=\hbox{$#1{#2#3}{\int}$ }
\vcenter{\hbox{$#2#3$ }}\kern-.58\wd0}}

\def\vint_#1{\mathchoice%
        {\mathop{\kern 0.2em\vrule width 0.6em height 0.69678ex depth -0.58065ex
                \kern -0.8em \intop}\nolimits_{\kern -0.4em#1}}%
        {\mathop{\kern 0.1em\vrule width 0.5em height 0.69678ex depth -0.60387ex
                \kern -0.6em \intop}\nolimits_{#1}}%
        {\mathop{\kern 0.1em\vrule width 0.5em height 0.69678ex
            depth -0.60387ex
                \kern -0.6em \intop}\nolimits_{#1}}%
        {\mathop{\kern 0.1em\vrule width 0.5em height 0.69678ex depth -0.60387ex
                \kern -0.6em \intop}\nolimits_{#1}}}
\def\vintslides_#1{\mathchoice%
        {\mathop{\kern 0.1em\vrule width 0.5em height 0.697ex depth -0.581ex
                \kern -0.6em \intop}\nolimits_{\kern -0.4em#1}}%
        {\mathop{\kern 0.1em\vrule width 0.3em height 0.697ex depth -0.604ex
                \kern -0.4em \intop}\nolimits_{#1}}%
        {\mathop{\kern 0.1em\vrule width 0.3em height 0.697ex depth -0.604ex
                \kern -0.4em \intop}\nolimits_{#1}}%
        {\mathop{\kern 0.1em\vrule width 0.3em height 0.697ex depth -0.604ex
                \kern -0.4em \intop}\nolimits_{#1}}}

\begin{document}

\maketitle

\begin{abstract}
Let $n\geq 2$ and $1\leq q<p<\fz$. We prove that if $\boz\subset\rn$ is a Sobolev $(p, q)$-extension domain, with additional capacitory restrictions on boundary in the case $q\leq n-1$, $n>2$, then $|\partial\boz|=0$. In the case $1\leq q<n-1$, we give an example of a Sobolev $(p,q)$-extension domain with $|\partial\boz|>0$.
\end{abstract}

\section{Introduction}
Let $1\leq q\leq p\leq\fz$. Then a bounded domain $\boz\subset\rn$, $n\geq 2$, is said to be a Sobolev $(p, q)$-extension domain if there exists a bounded extension operator
$$ E:W^{1,p}(\boz)\to W^{1,q}(\rn).$$ 

Partial motivation for the study of Sobolev extensions comes from PDEs (see, for example, \cite{M}). In \cite{C61, Stein} it was proved that if $\Omega\subset\mathbb R^n$ is a Lipschitz domain, then there exists a bounded linear extension operator $E: W^{k, p}(\Omega)\to W^{k, p}(\mathbb R^n)$,  for each $k\ge1$ and all $1\leq p\leq \infty$. Here $W^{k, p}(\boz)$ is the Banach space of $L^p$-integrable functions whose weak derivatives up to order $k$ belong to $L^p(\Omega)$.
More generally, the notion of $(\varepsilon,\delta)$-domains was introduced in \cite{J81} and it was proved that, for every $(\varepsilon,\delta)$-domain there exists a bounded linear extension operator $E: W^{k, p}(\Omega)\to W^{k, p}(\mathbb R^n)$, for all $k\geq 1$  and $1\leq p\leq\fz$. 

A geometric characterization of simply connected planar Sobolev $(2,2)$-extension domains was obtained in \cite{VGL79}. By later results in \cite{pekkaJFA, KRZ1, KRZ2, ShJFA}, we understand the geometry 
of simply connected planar Sobolev $(p, p)$-extension domains, for all $1\leq p\leq\fz$. Geometric characterizations are also known in the case of homogeneous Sobolev spaces $L^{k, p}(\Omega)$, $2<p<\infty$, 
defined on simply connected planar domains. Here $L^{k, p}(\boz)$ is the seminormed space of locally integrable functions whose $k$th-order distributional partial derivatives belong to $L^P(\boz)$. However, no characterizations are available in the general setting.

The boundary $\partial\boz$ of a Sobolev $(p, p)$-extension domain is necessarily of volume zero when $1\leq p<\fz$ by results in \cite{HKT}. Actually, $\boz$ has to be Ahlfors regular in the sense that 
\begin{equation}\label{eq:regular}
|B(x, r)\cap\boz|\geq C|B(x, r)|
\end{equation}
for every $x\in\partial \Omega$ and all $0<r<\min\{1,\frac{1}{4}\diam\boz\}$ with a constant $C$ independent of $x, r$. Even more is known if $\Omega$ is additionally a planar Jordan domain. In this case $\Omega$ 
has to be a so-called John domain when $1\leq p\leq 2$ and the complementary domain needs to be a John domain when $2\leq p<\fz$. Consequently, the Hausdorff dimension of $\partial \boz$ is necessarily strictly 
less than two by results in \cite{KR}. For a sharp estimate see the very recent paper \cite{LRT}. However, in general, the Hausdorff dimension of the boundary of a Sobolev $(p, p)$-extension domain $\boz\subset\rn$ 
can well be $n$.

Much less is known when $q<p$. First of all, no geometric criteria is available even when $\boz$ is planar and Jordan. The only existing result related to (\ref{eq:regular}) is the generalized Ahlfors-type estimate
\begin{equation}\label{eq:Qregu}
\Phi(B(x, r))^{p-q}|B(x,r)\cap\Omega|^q\geq C|B(x, r)|^p
\end{equation}
from \cite{ukhlov1} (also see \cite{ukhlov1'}) for the case $n<q<p<\fz$. Here $\Phi$ is a bounded, monotone and countably additive set function, defined on open sets $U\subset\rn$ with $U\cap\boz\neq\emptyset$. It is generated by the 
extension property. By differentiating $\Phi$ with respect to the Lebesgue measure, one concludes that $|\partial\boz|=0$ if $\boz$ is a Sobolev $(p, q)$-extension domain with $n<q<p<\fz$.

Our first result gives the optimal capacitory version of (\ref{eq:Qregu}) for the full scale $1\leq q<p<\fz$. Towards the statement, we set $A(x;s,t)=B(x,t)\setminus B(x,s)$ when $x\in \rn$ and $0<s<t.$

\begin{thm}\label{thm:capa}
Let $\boz\subset\rn$ be a Sobolev $(p, q)$-extension domain with $1\leq q<p<\fz$. Then
there exists a nonnegative, bounded, monotone and countably additive set function $\Phi$, defined on open sets, such that, for every $x\in\partial {\boz}$ and each $0<r<\min\{1, \frac{1}{4}\diam(\boz)\}$, we have
\begin{equation}\label{eq:capden}
\Phi(B(x, r))^{p-q}|B(x, r)\cap\boz|^q\geq r^{pq}Cap_q\lf(\boz\cap B\lf(x, \frac{r}{4}\r), \boz\cap A\lf(x; \frac{r}{2}, \frac{3r}{4}\r); B(x, r)\r)^p.
\end{equation}
Here $Cap_q$ is the classical variational $q$-capacity.
\end{thm}


Since the lower bound in (\ref{eq:capden}) comes with a term related to the capacitory size of a portion of $\boz$, let us analyze it carefully in the model case of an exterior spire of doubling order. More precisely, let $w:[0, \fz)\to[0, \fz)$ be continuous, increasing and differentiable with $w(0)=0$, $w(1)=1$ and so that $w(2t)\leq Cw(t)$ for all $t>0$. We also require $w'$ to be increasing on $(0,1)$ with $\lim_{t\to0^+}w'(t)=0.$ We define
\begin{equation}\label{eq:Cusp}
\boz^n_w:=\{z=(t, x)\in\rr\times\rr^{n-1}:0<t\leq 1, |x|<w(t)\}\cup B^{n}((2, 0), \sqrt{2}).
\end{equation}
See Figure $1$. We call $\boz^n_w$ an outward cusp domain with a doubling cusp function $w$. The boundary of $\boz^n_w$ contains an exterior spire of order $w$ at the origin.
\begin{figure}[h]
\centering
\includegraphics[width=8cm]{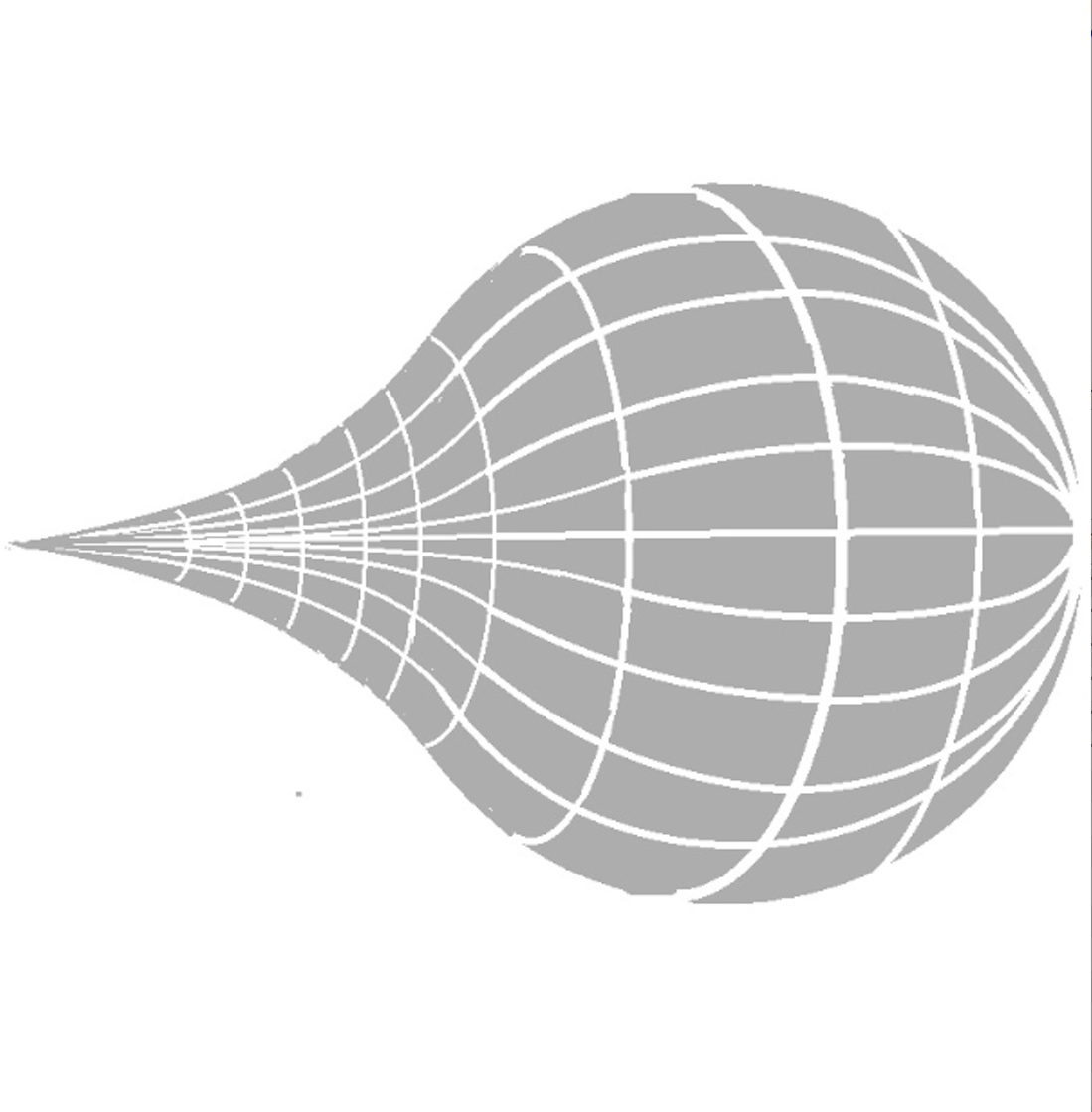}
\caption{An outward cuspidal domain.}
\label{fig:cups}
\end{figure}

We write $A\sim_cB$ if $\frac{1}{c}A\leq B\leq cA$ for a constant $c>1$. The following theorem gives the sharp capacitory estimate at the origin for the outward cusp domain $\boz^n_w$.

\begin{thm}\label{lem:cacusp}
Let $n\geq 3,$ and let $\boz^n_{w}\subset\rn$ be an outward cusp domain with a doubling cusp function $w$. Then, for every $0<r<1$, we have
\begin{multline}\label{eq:cacuspH}
Cap_p\lf(\boz^n_w\cap B\lf(0,\frac{r}{4}\r),\boz^n_{w}\cap A\lf(0;\frac{r}{2}, \frac{3r}{4}\r); B(0, r)\r)\sim_c\\
Cap_p\lf(\boz^n_w\cap B\lf(0,\frac{r}{4}\r), A\lf(0; \frac{r}{2}, \frac{3r}{4}\r); B(0, r)\r)\sim_c\\\begin{cases}
r^{n-p} &\textnormal{if }\ n-1<p<\fz\\
\frac{r}{\log^{n-2}\frac{r}{w(r)}} &\textnormal{if }\ p=n-1\\
r(w(r))^{n-1-p} &\textnormal{if }\ 1\leq p<n-1\, \end{cases}
\end{multline}
where the constant $c$ is independent of $r$.
\end{thm}

The $(p,q)$-extendability properties for the domains $\boz^n_{w}$ are known by 
\cite{Mazya1,Mazya2,Mazya,Mazya3}.
We show in Section 5 that these domains give examples of settings where the exponents in \eqref{eq:capden} are optimal and where boundedness of $\Phi$ cannot be replaced,
say, by an estimate of the type $\Phi(B(x,r))\le Cr^\alpha.$ Besides of boundedness, the other crucial property of our set function $\Phi$ is additivity. It allows one to obtain better volume estimates when the center of $B(x,r)$ does not belong to a suitable exceptional set.  These estimates are shown to be sharp in Section 5 for wedges generated by $\boz^n_{w}.$ 

In order to effectively use \eqref{eq:capden}, one needs an estimate for the respective capacitory term.
For this, we employ the notion of a $q$-capacitory dense domain. Roughly, a domain $\boz\subset\rn$ is $q$-capacitory dense at a point $x\in\rn$, if there exists a decreasing sequence $\{r_i\}$ converging to zero so that the $q$-capacity of $\boz\cap B\lf(x, \frac{r_i}{4}\r)$ and of $\boz\cap A\lf(x; \frac{r_i}{2}, \frac{3r_i}{4}\r)$ in $B(x, r_i)$ are comparable to the $q$-capacity of the pair $B\lf(x, \frac{r_i}{4}\r)$ and $A\lf(x;\frac{r_i}{2}, \frac{3r_i}{4}\r)$ in $B(x, r_i)$ with an absolute constant.  We say that a condition holds for almost every $x\in \partial \boz$ if there is a set $E\subset \partial \Omega$ of volume zero so that the condition holds on $\partial \boz \setminus E.$

We deduce the following generalized Ahlfors-type measure density estimate from (\ref{eq:capden}).

\begin{cor}\label{thm:poicade}
Let $\boz\subset\rn$ be a Sobolev $(p, q)$-extension domain which is $q$-capacitory dense at almost every $x\in\partial\boz$, where $1\leq q<p<\fz$. Then there exists  a nonnegative, bounded, monotone and countably additive set function $\Phi,$ defined on open sets, with the following property. For almost every $x\in\partial\boz$, we have
\begin{equation}
\limsup_{r\to 0^+}\frac{\Phi(B(x,r))^{p-q}|B(x,r)\cap\boz|^q}{|B(x,r)|^p}>0.\nonumber
\end{equation}
\end{cor}

In potential theory, the concept of fatness often leads to sharper results than the notion of capacity density. Roughly, a domain $\boz$ is $q$-fat at a point $x\in\rn$, if the $q$-capacity of $\boz\cap B(x, r)$ is not very small in average, when compared to the $q$-capacity of the ball $B(x, r).$ The following observation shows that $q$-capacitory density implies $q$-fatness. 

\begin{prop}\label{le:stro}
Let $n\geq 3$. If a domain $\boz\subset\rn$ is $q$-capacitory dense at a point $x\in\rn$ for some $1\leq q<\fz$, then $\boz$ is also $q$-fat at $x$. On the other hand, for arbitrary $1<q\leq n-1$, there exists $w$ so that the domain $\boz^n_w$  is $q$-fat but not $q$-capacitory dense at the origin. 
\end{prop}

Consequently, Corollary \ref{thm:poicade} is also a corollary to the following stronger result. 

\begin{thm}\label{thm:point}
Let $\boz\subset\rn$ be a Sobolev $(p, q)$-extension domain which is $q$-fat at almost every $x\in\partial\boz$, where $1\leq q< p<\fz$. Then there exists a nonnegative, bounded, monotone and countably additive set function $\Phi,$ defined on open sets, with the following property. For almost every $x\in\partial\boz$, there exists $r_x>0$ such that, for every $0<r<r_x$, we have
\begin{equation}\label{equa:meade}
\Phi(B(x,r))^{p-q}|B(x,r)\cap\boz|^q\geq |B(x,r)|^p.
\end{equation}
\end{thm}

\vskip 0.3 cm

Our next result clarifies the role of $q$ in the validity of fatness and capacitory density.

\begin{thm}\label{thm:q-fat}
Let $\Omega\subset \rr^2$ be a domain. Then $\Omega$ is $q$-capacitory dense at each $x\in\partial\boz$, for every $1\leq q<\fz$. 

Let $n\geq 3$ and let $\Omega\subset\rn$ be a domain. Then $\Omega$ is $q$-capacitory dense at each $x\in\partial\boz$ when $n-1<q<\fz$. Conversely, if $1\leq q\le n-1$, then there exists a domain $\Omega\subset\rn$ with $|\partial\boz|>0$ such that $\boz$ fails to be $q$-fat at points of a subset of positive volume of $\partial\boz$.
\end{thm}

 By differentiating our additive set function $\Phi$ with respect to the Lebesgue measure, Theorem \ref{thm:point} together with  Theorem \ref{thm:q-fat} and  the Lebesgue density theorem yield the following conclusion on the volume of the boundary of a Sobolev $(p,q)$-extension domain.
 
\begin{thm}\label{thm:high}
Let $\boz\subset\rn$ be a Sobolev $(p, q)$-extension domain with $1\leq q< p<\fz$. Under the assumption of Theorem \ref{thm:point}, we have $|\partial\boz|=0$. In particular, if $\boz\subset\rr^2$ is a Sobolev $(p, q)$-extension domain with $1\leq q<p<\fz$, then $|\partial\boz|=0$. Moreover, if $n\geq 3$ and $\boz\subset\rn$ is a Sobolev $(p, q)$-extension domain with $n-1<q<p<\fz$, then $|\partial\boz|=0$.
\end{thm}

We have not required our Sobolev extension operators to have any local properties. Let us consider such a requirement.
Let $\boz\subset\rn$ be a bounded Sobolev $(p, q)$-extension domain with $1\leq q<p<\fz$. Then a bounded extension operator $E_s:W^{1, p}(\boz)\to W^{1,q}(\rn)$ is said to be a strong extension operator, if for every function $u\in W^{1,p}(\boz)\cap C(\boz)$ with $u\big|_{B(x, r)\cap\boz}\equiv c$ for some ball $B(x, r)$ intersecting $\boz$ and some constant $c\in\rr$, we have $E_s(u)(y)=c$ for almost every $y\in B(x, r)\cap\partial\boz$. 

The following theorem shows that an extension operator can be promoted to a strong one precisely when the boundary of our extension domain is of volume zero. 

\begin{thm}\label{thm:strong}
Let $\boz\subset\rn$ be a bounded Sobolev $(p, q)$-extension domain with $1\leq q\leq p<\fz$. Then there exists a strong extension operator $E_s:W^{1,p}(\boz)\to W^{1,q}(\rn)$ if and only if $|\partial\boz|=0$.
\end{thm}

Recall that a function $u:\Omega\to\mathbb R$ is said to be $ACL(\boz)$, if it is absolutely continuous on almost all line segments parallel to coordinate axes. According to the Tonelli characterization of Sobolev functions, a Sobolev function can be redefined on a set of measure zero so as to belong to $ACL(\boz)$ \cite{M}. Let $\boz\subset\rn$ be a Sobolev $(p, q)$-extension domain and $E: W^{1,p}(\boz)\to W^{1,q}(\rn)$ be the corresponding bounded extension operator. Suppose that $u\in W^{1,p}(\boz)$ satisfies $u\equiv 1$ on $U\cap\boz$ for some open set $U\subset\rn$ such that $U\cap\partial\boz\neq\emptyset$. By the Tonelli characterization, it is natural to expect that $E(u)(x)=1$ should hold for almost every $x\in U\cap\overline\boz$. Hence, $E$ should be a strong extension operator. However, our next theorem shows that this is not always the case.  

\begin{thm}\label{thm:positive}
Let $n\geq3$ and $1\leq q<n-1$. Then there exists $p>q$ and a Sobolev $(p, q)$-extension domain $\boz\subset\rn$ with $|\partial\boz|>0$.
\end{thm}

Corollary \ref{thm:poicade} relies on the assumption that, for almost every $x\in\partial\boz$, we have the capacity estimate
$$
Cap_q\lf(\boz\cap B\lf(x, \frac{r_i}{4}\r), \boz\cap A\lf(x; \frac{r_i}{2}, \frac{3r_i}{4}\r); B(x, r_i)\r)\geq C(x)r_i^{n-q},
$$
for a decreasing sequence $\{r_i\}_{i=1}^\fz$ which tends to zero. For $n-1<q<\fz$, this is always the case for an arbitrary domain $\boz\subset\rn$. 
On the other hand,
this estimate may fail miserably when $1\le q<n-1.$ 

\begin{thm}\label{thm:caF}
Let $n\geq 3$ and $1\leq q<n-1$ be arbitrary, and let $h:[0, 1]\to[0, 1]$ be a strictly increasing and continuous function with $h(0)
=0$ and $h(1)=1$. Then there exists $q<p<\fz$ and a Sobolev $(p, q)$-extension domain $\boz_h\subset\rn$ with a subset $A\subset\partial\boz_h$ of positive volume so that
$$
\lim_{r\to0^+}\frac{Cap_q\lf(\boz_h\cap B\lf(x, \frac{r}{4}\r), \boz_h\cap A\lf(x; \frac{r}{2}, \frac{3r}{4}\r); B(x, r)\r)}{h(r)}=0
$$
for every $x\in A$.
\end{thm}




This paper is organized as follows. Section 2 contains definitions and preliminary results. We reduce Theorem \ref{thm:q-fat} to our other results in Section 3. Section 4 contains the proofs of Corollary \ref{thm:capa} and Theorems \ref{thm:poicade}, \ref{thm:point},  \ref{thm:high} and \ref{thm:strong}.
We prove Theorem \ref{lem:cacusp} in Section 5 and Proposition \ref{le:stro} in Section 6.
Section 7 is devoted to the construction behind Theorems  \ref{thm:positive} and \ref{thm:caF}.
In the final section, Section 8, we pose open problems that arise from the results in this paper
and discuss the locality of our estimates.

\section{Preliminaries}

\subsection{Definitions and notation}
For a function $u\in L^1_{\rm loc}(\boz)$ and a measurable set $A\subset\boz$ with $|A|>0$, 
\[u_A:=\bint_Au(x)dx:=\frac{1}{|A|}\int_Au(x)dx\] 
means the integral average of $u$ over the set $A$.

Let $\Omega$ be a domain in the $n$-dimensional Euclidean space $\rn$ with $n\geq 2$. By the symbol $\operatorname{Lip}(\boz)$ we denote the class of all Lipschitz continuous functions defined on $\boz$. The Sobolev space $W^{1,p}(\Omega)$, $1\leq p\leq\infty$, (see, for example, \cite{M}) is defined 
as a Banach space of locally integrable and weakly differentiable functions
$u:\Omega\to\mathbb{R}$ equipped with the norm: 
\[
\|u\|_{W^{1,p}(\Omega)}=\| u\|_{L^p(\Omega)}+\|\nabla u\|_{L^p(\Omega)},
\]
where $\nabla u=\left(\frac{\partial u}{\partial x_1},...,\frac{\partial u}{\partial x_n}\right)$ is a weak gradient of $u$.

Let us give the definition of Sobolev extension domains.

\begin{defn}\label{defnsoex}
Let $1\leq q\leq p<\fz$. A bounded domain $\boz\subset\rn$ is said to be a Sobolev $(p,q)$-extension domain, if there exists a bounded operator 
\begin{equation}
E: W^{1,p}(\boz)\to W^{1,q}(\rn)\nonumber
\end{equation}
such that for every function $u\in W^{1,p}(\boz)$, the function $E(u)\in W^{1,q}(\rn)$ satisfies $E(u)\big|_{\boz}\equiv u$ and 
$$
\|E\|:=\sup_{u\in W^{1,p}(\boz)\setminus\{0\}}\frac{\|E(u)\|_{W^{1,q}(\rn)}}{\|u\|_{W^{1,p}(\boz)}}<\infty.
$$
\end{defn}

 Sobolev $(p,q)$-extension operators arise as extension operators in non-Lipschitz domains, see \cite{GS, Mazya1, Mazya2, Mazya, Mazya3}. The outward cusp domains $\boz^n_{t^s}$ from our introduction are standard examples of Sobolev $(p, q)$-extension domains with $q$ strictly less than $p$. 
The optimal Sobolev extension pairs $(p, q)$ for these domains are known due to Maz'ya and Poborchi, \cite{Mazya1, Mazya2, Mazya, Mazya3}. 


In general, we work with potentially non-linear extension operators, but we distinguish homogeneous Sobolev extension operators.

\begin{defn}\label{defn:homo}
Let $E: W^{1,p}(\boz)\to W^{1,q}(\rn)$ be a bounded Sobolev $(p, q)$-extension operator with $1\leq q\leq p<\fz$. We say that it is a homogeneous extension operator if for every $u\in W^{1,p}(\boz)$ and $\lambda\in\rr$, $E(\lambda u)(x)=\lambda E(u)(x)$ holds, for every $x\in\rn$.
\end{defn}

Linear Sobolev extension operators form a subclass of homogeneous Sobolev extension operators. We prove that, for every Sobolev $(p, q)$-extension domain, there always exists a homogeneous Sobolev extension operator. When $q=p$ one in fact can find a linear extension operator \cite{HKT}
but it is not known if this could be the case when $q<p.$

\begin{lem}\label{lem:homo}
Let $\boz\subset\rn$ be a Sobolev $(p, q)$-extension domain. Then every bounded Sobolev extension operator $E:W^{1,p}(\boz)\to W^{1,q}(\rn)$ promotes to a bounded homogeneous Sobolev extension operator $E_h:W^{1,p}(\boz)\to W^{1,q}(\rn)$ with the operator norm inequality $\|E_h\|\leq \|E\|$.
\end{lem}

\begin{proof}
Let $u\in W^{1,p}(\boz)$ be arbitrary with $u\neq 0$. Then we define 
$$E_h(u):=\|u\|_{W^{1,p}(\boz)}E\lf(\frac{u}{\|u\|_{W^{1,p}(\boz)}}\r).$$
For $u=0$, we simply set $E_h(u)=E(u)=0$. Then, for any function $u\in W^{1,p}(\boz)$ and $\lambda\in\rr$, we have $E_h(\lambda u)=\lambda E_h(u)$. Moreover,
\begin{multline}
\|E_h\|:=\sup_{u\in W^{1,p}(\boz)\setminus\{0\}}\lf\|E\lf(\frac{u}{\|u\|_{W^{1,p}(\boz)}}\r)\r\|_{W^{1,q}(\rn)}
                  =\sup_{\|u\|_{W^{1,p}(\boz)}=1}\|E(u)\|_{W^{1,q}(\rn)}\nonumber\\
                  \leq\sup_{u\in W^{1,p}(\boz)\setminus\{0\}}\frac{\|E(u)\|_{W^{1,q}(\rn)}}{\|u\|_{W^{1,p}(\boz)}}=:\|E\|.\nonumber
\end{multline}
\end{proof}

By Lemma \ref{lem:homo}, from now on, we may always assume that $E:W^{1,p}(\boz)\to W^{1,q}(\rn)$ is a homogeneous bounded Sobolev extension operator.     

We continue with the definition of a strong bounded Sobolev extension operator.

\begin{defn}\label{defn:strong}
Let $\boz\subset\rn$ be a Sobolev $(p, q)$-extension domain with $1\leq q\leq p<\fz$. A bounded Sobolev extension operator $E_s:W^{1,p}(\boz)\to W^{1,q}(\rn)$ is said to be a strong bounded Sobolev extension operator if, for every function $u\in W^{1,p}(\boz)$ with $u\big|_{B(x, r)\cap\boz}\equiv c$ for some ball $B(x, r)$ with $B(x, r)\cap\boz\neq\emptyset$ and some constant $c\in\rr$, we have $E_s(u)(y)=c$ for almost every $y\in B(x, r)\cap\partial\boz$.
\end{defn}



\subsection{Fine Topology}

In this section, we recall some basic facts about the fine topology on $\rn$. It is the coarsest topology on $\rn$ in which all superharmonic functions on $\rn$ are continuous, see \cite[Chapter 12]{HKM:book}.

Let us recall the notion of variational $p$-capacity \cite{GResh, HKM:book, M}.

\begin{defn}\label{defn:cap}
A condenser in a domain $\boz\subset\rn$ is a pair $(E, F)$ of bounded subsets of $\overline\boz$ with $\dist(E, F)>0$. Fix $1\leq p<\fz$. The set of admissible functions for the triple $(E,F;\Omega)$ is
\[\mathcal W_p(E, F; \boz)=\{u\in W^{1,p}(\boz)\cap C(\boz\cup E\cup F): u\geq 1\ {\rm on}\ E\ {\rm and}\ u\leq 0\ {\rm on}\ F\}.\]
We define the $p$-capacity of the pair $(E, F)$ with respect to $\boz$ by setting. 
\[Cap_p(E, F;\boz)=\inf_{u\in\mathcal W_p(E, F; \boz)}\int_\boz|\nabla u(x)|^pdx.\]
\end{defn}

The following lemma gives the basic Teichm\"uller-type capacity estimate. The interested readers can find a proof in \cite{HKacta}.

\begin{lem}\label{lem2.1}
Let $B\subset\rr^n$ be a ball with radius $r$ and $n-1<p<\fz$. Suppose that $E, F\subset B$ are connected subsets with $\dist(E, F)>0$ and so that $\diam E\geq\delta r$ and $\diam F\geq\delta r$ for some $0<\delta<2$. Then we have
\begin{equation}\label{eq2.1}
Cap_p(E,F;B)\geq Cr^{n-p},
\end{equation}
 where the constant $C$ only depends on $\delta,$ $n$ and $p$. The inequality also holds for $p=1$ when $n=2$.
\end{lem}

We have the following capacity estimate for concentric balls. See, for example, \cite[page 35]{HKM:book}.
\begin{equation}\label{equa:cap}
Cap_p(B(x,r), A(x; R, 2R); B(x, 2R))=\begin{cases}\omega_{n-1}\lf(\frac{|n-p|}{p-1}\r)^{p-1}\lf|R^{\frac{p-n}{p-1}}-r^{\frac{p-n}{p-1}}\r|^{1-p} & p\neq n\\
\omega_{n-1}\log^{1-n}\frac{R}{r} & p=n\, \end{cases}
\end{equation}
where $0<r<R<\fz$ and $\omega_{n-1}$ means the $(n-1)$-dimensional volume of the unit sphere $S^{n-1}(0, 1)$. 

Now, we are ready to define quasi-continuous functions, see \cite{HKM:book, M}.
\begin{defn}\label{defn:quasicon}
Let $1\leq p<\fz$. A function $u\in L^1_{\rm loc}(\boz)$ is said to be $p$-quasi-continuous if, for every $\epsilon>0$, there exists a set $E_\epsilon\subset\boz$ with $Cap_p(E_\epsilon,\partial\boz; \boz)<\epsilon$ such that $u\big|_{\boz\setminus E_\epsilon}$ is continuous.
\end{defn}

We record the fact that every Sobolev function can be redefined in a set of measure zero so as to become quasi-continuous. See \cite[Chapter 4]{HKM:book} or \cite{M}.

\begin{lem}\label{lem:quasire}
Let $1\leq p<\fz$ and let $u\in W^{1,p}(\boz)$. Then there exists a $p$-quasi-continuous function $\tilde u\in W^{1,p}(\boz)$ with $\tilde u(z)=u(z)$ for almost every $z\in\boz$. Furthermore, at every point of continuity of $u$, we have $\tilde u(z)=u(z)$.
\end{lem}

We continue with the definition of $p$-capacitory density.
\begin{defn}\label{de:capden}
Let $1\leq p<\fz$. A set $E\subset\rn$ is said to be $p$-capacitory dense at the point $z\in\rn$, if 
\[\limsup_{r\to0^+}\frac{Cap_p\lf(E\cap B\lf(z, \frac{r}{4}\r), E\cap A\lf(z; \frac{r}{2}, \frac{3r}{4}\r); B(z, r)\r)}{Cap_p\lf(B\lf(z,\frac{r}{4}\r), A\lf(z; \frac{r}{2}, \frac{3r}{4}\r); B(z, r)\r)}>0.\]
\end{defn}

The following definition of $p$-thin sets can be found in \cite[Chapter 12]{HKM:book} for $1<p<\fz$ and in \cite{panu1} for $p=1$. 
\begin{defn}\label{defn:p-thin}
Let $1<p<\fz$. A set $E$ is $p$-thin at $x$ if 
\[\int_0^1\lf(\frac{Cap_p(E\cap B(x, t),A(x; 2t, 3t); B(x, 4t))}{Cap_p(B(x, t), A(x; 2t, 3t); B(x, 4t))}\r)^{\frac{1}{p-1}}\frac{dt}{t}<\fz.\]
A set $E$ is $1$-thin at $x$ if 
\[\lim_{t\to 0}t\frac{Cap_1(E\cap B(x, t), A(x; 2t,3t); B(x, 4t))}{\mathcal H^n(B(x, t))}=0.\]
Furthermore, we say that $E$ is $p$-fat at $x$ if $E$ is not $p$-thin at $x$.
\end{defn}



\begin{defn}\label{defn:Ftopology}
Let $1\leq p<\fz$. A set $U\subset\rn$ is $p$-finely open if $\rn\setminus U$ is $p$-thin at every $x\in U$, and
\[\tau_p:=\{U\subset\rn; U\ {\rm is}\ p-{\rm finely}\ {\rm open}\}\]
is the $p$-fine topology on $\rn$.
\end{defn} 
The following lemma comes from \cite[Corollary 12.18]{HKM:book}.
\begin{lem}\label{lem:FC}
Suppose that a set $E\subset\rn$ is $p$-fat at the point $x\in\rn$. Then every $p$-finely open neighborhood of $x$ intersects $E$. Consequently, $x$ is a  $p$-fine limit point of $E$. 
\end{lem}
By a result due to Fuglede \cite{bent}, we have the following lemma.
\begin{lem}\label{lem:qsandfi}
Let $1\leq p<\fz$. If a function $u$ is $p$-quasi-continuous, then $u$ is $p$-finely continuous except on a subset of $p$-capacity zero.
\end{lem}

By using the lemmata above, we can prove the following lemma. It is also a corollary of the result in \cite{Tero}.

\begin{lem}\label{lem:u=1}
Let $\boz\subset\rn$ be a domain such that $\boz$ is $p$-fat at almost every point of the boundary $\partial\boz$. If $u\in W^{1,p}(\rn)$ is a Sobolev function such that $u\big|_{B(x, r)\cap\boz}\equiv c$, where $x\in\partial\boz$, $0<r<1$ and $c\in\rr$, then $u(z)=c$ for almost every $z\in\partial\boz\cap B(x, r)$.
\end{lem}

\begin{proof}
By Lemma \ref{lem:quasire}, $u$ has a $p$-quasi-continuous representative $\tilde u$ with $\tilde u(z)=c$ for every $z\in\boz\cap B(x, r)$. By Lemma \ref{lem:qsandfi} and \cite[Theorem 4.17]{Evans}, there exists a subset $E_1\subset\rn$ with $|E_1|=0$ such that $\tilde u$ is $p$-finely continuous on $\rn\setminus E_1$. Since $\boz$ is $p$-fat at almost every $z\in\partial\boz$, by Lemma \ref{lem:FC}, there exists a subset $E_2\subset\partial\boz$ with $|E_2|=0$ such that, for every $z\in (\partial\boz\cap B(x, r))\setminus(E_1\cup E_2)$, we have $\tilde u(z)=c$. Hence $u(z)=c$ for almost every $z\in\partial\boz\cap B(x,r)$.
\end{proof} 

\subsection{ A set function associated with the extension operator}

In this section, we will discuss the notion of additive set functions (outer measures), associated with bounded, homogeneous Sobolev extension operators, as introduced by Ukhlov in \cite{ukhlov1, ukhlov1'}. Also
see \cite{ukhlov2,ukhlov3} for related set functions.

Let $\boz\subset\rr^n$ be a bounded Sobolev $(p, q)$-extension domain and $ E:W^{1,p}(\boz)\to W^{1,q}(\rr^n)$ be the corresponding bounded and homogeneous extension operator with $1\leq q<p<\fz$. 
Suppose that $U\subset\rr^n$ is an open set such that $U\cap\boz\neq\emptyset$. Then we denote by $W^p_0(U, \boz)$ the class of continuous  functions $u\in W^{1,p}(\Omega)$ such that $u\eta$ belongs to $W^{1,p}(U\cap\Omega)\cap C_0(U\cap\Omega)$ for all smooth functions $\eta\in C_0^{\infty}(\Omega)$ (roughly, it is a class of continuous functions $u\in W^{1,p}(\boz)$ such that $u(x)=0$ for every $x\in\boz\setminus U$).

The set function $\Phi$  is defined by setting
 \begin{equation}\label{setfunc11}
 \Phi(U):=\sup_{u\in W^p_0(U, \boz)}\lf(\frac{\|\nabla E(u)\|_{L^{q}(U)}}{\|u\|_{W^{1,p}(U\cap\boz)}}\r)^{\kappa},\ \ \frac{1}{k}=\frac{1}{q}-\frac{1}{p},
 \end{equation}
 for every open set $U\subset\rn$ that intersects $\boz$ and by setting $\Phi(U)=0$ for those open
 sets that do not intersect $\boz.$
 
Our set function is a modification of the set function in \cite{ukhlov1, ukhlov1'}. 
The following theorem gives the important properties of $\Phi$.

\begin{thm}\label{thm:sofun1}
Let $1\leq q<p<\fz$. Let $\boz\subset\rr^n$ be  a bounded Sobolev $(p, q)$-extension domain and $ E:W^{1,p}(\boz)\to W^{1,q}(\rr^n)$ be the corresponding homogeneous bounded extension operator. Then the set function $\Phi$ defined in (\ref{setfunc11}) is a nonnegative, bounded, monotone and countably additive set function defined on open subsets $U\subset\rr^n$.
\end{thm}

 The proof of this theorem repeats the proof of \cite[Theorem 2.1]{ukhlov1'} with  minor modifications.
Since we do not assume linearity of $E$, we give the details for the sake of completeness.

\begin{proof}
The nonnegativity, boundedness and monotonicity of $\Phi$ are immediate from the definition.  Hence it suffices to prove the additivity of $\Phi.$

Let $\{U_j\}_{j=1}^\fz$ be a sequence of pairwise disjoint open sets in $\rn.$ We may assume that $U_j\cap\boz\neq\emptyset$ for each $j.$ We define $U_0:=\bigcup_{j=1}^\fz U_j$. Then, for every $j\in\mathbb N$, since $E$ is a bounded homogeneous Sobolev extension operator, we can choose a test function $u_j\in W^p_0(U_j, \boz)$ such that 
\begin{equation}\label{tarvitaan}
\|\nabla E(u_j)\|_{L^q(U_j)}\geq\lf(\Phi(U_j)\lf(1-\frac{\epsilon}{2^j}\r)\r)^{\frac{1}{{\kappa}}}\|u_j\|_{W^{1,p}(U_j\cap\boz)}
\end{equation}
and 
\begin{equation}\label{equa:CONS}
\|u_j\|^p_{W^{1,p}(U_j\cap\boz)}=\Phi(U_j)\lf(1-\frac{\epsilon}{2^j}\r),
\end{equation}
where $\epsilon \in (0,1)$ is fixed. Set $v_N:=\sum_{j=1}^Nu_j.$ Then $v_N\in W^p_0(\bigcup_{j=1}^NU_j, \boz).$ By \eqref{tarvitaan} and \eqref{equa:CONS} we have
\begin{multline} \label{tagi}
\|\nabla E(v_N)\|_{L^q(\bigcup_{j=1}^NU_j)}
\geq \lf(\sum_{j=1}^N\lf(\Phi(U_j)\lf(1-\frac{\epsilon}{2^j}\r)\r)^{\frac{q}{{\kappa}}}\left\|u_j\right\|^q_{W^{1, p}(U_j\cap\boz)}\r)^{\frac{1}{q}}=\\
                                                     \lf(\sum_{j=1}^N\Phi(U_j)\lf(1-\frac{\epsilon}{2^j}\r)\r)^{\frac{1}{{\kappa}}}\|v_N\|_{W^{1, p}\lf(\bigcup_{j=1}^NU_j\cap\boz\r)}
                                                     \geq\lf(\sum_{j=1}^N\Phi(U_j)-\epsilon\Phi(U_o)\r)^{\frac{1}{{\kappa}}}\|v_N\|_{W^{1,p}(\bigcup_{j=1}^NU_j\cap\boz)}.
\end{multline}
Since $v_N\in W^p_0(U_0, \boz)$, we conclude from \eqref{tagi} that  
$$
\Phi(U_0)^{\frac{1}{{\kappa}}}\geq\frac{\|\nabla E(v_N)\|_{L^q\lf(\bigcup_{j=1}^NU_j\r)}}{\|v_N\|_{W^{1,p}(\bigcup_{j=1}^NU_j\cap\boz)}}\geq \lf(\sum_{j=1}^N\Phi(U_j)-\epsilon\Phi(U_o)\r)^{\frac{1}{{\kappa}}}.\nonumber
$$
By letting  first $\epsilon$ tend to zero and then using nonnegativity and monotonicity of $\Phi$ we arrive at  
$$\sum_{j=1}^\fz\Phi(U_j)\leq \Phi\lf(\bigcup_{j=1}^\fz U_j\r).$$

Towards the opposite inequality, we fix $\epsilon>0$ and pick $u\in W^p_0(U_0, \boz)$ such that
$$\|\nabla E(u)\|_{L^q(U_0)}\geq \lf(\Phi(U_0)(1-\epsilon)\r)^{\frac{1}{{\kappa}}}\|u\|_{W^{1,p}(U_0\cap\boz)}.$$
Given $j\in\mathbb N$, we define $u_j:=u\big|_{U_j\cap\boz}$. Since $u_j\in W^p_0(U_j,\boz),$ we have 
$$
\sum_{j=1}^\fz\Phi(U_j)\geq\sum_{j=1}^\fz\frac{\|\nabla E(u_j)\|_{L^q(U_j)}}{\|u_j\|_{W^{1,p}(U_j\cap\boz)}}\nonumber\\
                                \geq\frac{\lf(\sum_{j=1}^\fz\|\nabla E(u_j)\|_{L^q(U_j)}\r)^{\frac{1}{q}}}{\lf(\sum_{j=1}^\fz\|u_j\|^p_{W^{1, p}(U_j\cap\boz)}\r)^{\frac{1}{p}}}\geq\lf(\Phi(U_0)(1-\epsilon)\r)^{\frac{1}{{\kappa}}}.\nonumber
$$
Since $\epsilon$ is arbitrary, we conclude that
$$\sum_{j=1}^\fz\Phi(U_j)\geq \Phi\lf(\bigcup_{j=1}^\fz U_j\r).$$
\end{proof}

The following corollary is immediate from the definition (\ref{setfunc11}) of the set function $\Phi$.

\begin{cor}\label{cor:sofun1}
Let $1\leq q<p<\fz$. Let $\boz\subset\rr^n$ be a bounded Sobolev $(p, q)$-extension domain and $\Phi$ be the set function from \eqref{setfunc11}, associated to the corresponding homogeneous extension operator $E$. Then, for every open set $U$ with $U\cap\boz\neq\emptyset$ and each function $u\in W^p_0(U, \boz)$, we have 
\begin{equation}\label{equa:E11}
\|\nabla E(u)\|_{L^{q}(U)}\leq\Phi^{\frac{1}{\kappa}}(U)\|u\|_{W^{1,p}(U\cap\boz)},\,\, {\rm where}\,\, {1}/{\kappa}={1}/{q}-{1}/{p}.
\end{equation}
\end{cor}


We define the upper volume derivative $\overline D\Phi$ at a point $x\in\overline\boz$ by setting
\begin{equation}
\nonumber
\overline D\Phi(x):=\limsup_{r\to 0^+}\frac{\Phi(B(x, r))}{|B(x, r)|}.
\end{equation}

The following lemma, proved in \cite{Rado, ukhlov2}, gives the upper differentiability of $\Phi$ with respect to the Euclidean volume.

\begin{lem}\label{lem:upfinite}
Let $\Phi$ be a nonnegative, bounded, monotone and countably additive set function defined on open subsets $U\subset\rr^n$. Then $\overline D\Phi(x)<\fz$ for almost every $x\in\overline\boz$.
\end{lem} 

\subsection{Gromov hyperbolicity}
For each $1\leq q<n-1$, we will construct a Sobolev $(p, q)$-extension domain whose boundary is of positive volume. In order to establish the extension property of the domain, we will employ an approximation argument. Our domain $\boz$ turns out to be $\delta$-Gromov hyperbolic with respect to the quasihyperbolic metric, which implies that $W^{1,\fz}(\boz)$ is dense in $W^{1,p}(\boz)$.
\begin{defn}\label{de:quasihy}
Let $\boz\subsetneqq\rn$ be a domain. Then the associated quasihyperbolic distance between a pair of points $x, y\in\boz$ is defined as 
$$\dist_{qh}(x, y)=\inf_{\gamma}\int_\gamma\frac{dz}{\dist(z, \partial\boz)},$$
where the infimum is taken over all the rectifiable curves $\gamma\subset\boz$ connecting $x$ and $y$. A curve attaining this infimum is called a quasihyperbolic geodesic between $x$ and $y$. The distance between two sets is also defined in a similar manner.
\end{defn}

The existence of quasihyperbolic geodesics comes from a result by Gehring and Osgood \cite{GandO}. We continue with the definition of Gromov hyperbolicity with respect to the quasihyperbolic metric. 

\begin{defn}\label{de:grohy}
Let $\delta>0$. A domain is called $\delta$-Gromov hyperbolic with respect to the quasihyperbolic metric, if for all $x, y, z\in\boz$ and every corresponding quasihyperbolic geodesic $\gamma_{x, y}$, $\gamma_{y, z}$ and $\gamma_{x,z}$, we have 
$$\dist_{qh}(w, \gamma_{y, z}\cup\gamma_{x,z})\leq\delta,$$
for arbitrary $w\in\gamma_{x, y}$.
\end{defn}

Let us give the definition of quasiconformal mappings.
\begin{defn}\label{de:quasico}
Let $\boz, \boz'$ be domains in $\rn$ and let $1\leq K<\fz$. A homeomorphism $f:\boz\to\boz'$ of the class $W^{1,n}_{\rm loc}(\boz,\rn)$ is said to be a $K$-quasiconformal mapping, if  
$$|Df(x)|^n\leq KJ_f(x),\,\,\text{for almost every}\,\, x\in\boz.$$
Here $|Df(x)|$ means the operator norm of the matrix $Df(x)$ and $J_f(x)$ is its Jacobian determinant.
\end{defn}

The following result was proved in \cite{BHK}.

\begin{lem}\label{lem:qbgh}
Let $\boz\subset\rn$ be a domain which is quasiconformally equivalent to the unit ball. Then $\boz$ is $\delta$-Gromov hyperbolic with respect to the quasihyperbolic metric, where $\delta>0$ depends only on the quasiconformality constant $K$ and $n$.  
\end{lem}
The following density result comes from \cite{KRZ}.
\begin{lem}\label{lem:density}
If $\boz\subset\rn$ is a bounded domain that is $\delta$-Gromov hyperbolic with respect to the quasihyperbolic metric, then, for every $1\leq p<\fz$, $W^{1, \fz}(\boz)$ is dense in $W^{1,p}(\boz)$.
\end{lem}

\section{Proof of Theorem \ref{thm:q-fat}}

In this section, we prove Theorem \ref{thm:q-fat}. 
Given $1\le q<n-1,$ Theorem \ref{thm:caF} gives a domain whose boundary is not $q$-fat at points of a subset of positive volume. Since the construction can be easily 
modified so as to also cover the case $q=n-1,$ see Remark \ref{rajatapaus},  we only prove the positive part of Theorem \ref{thm:q-fat}.

\begin{thm}
Let $\Omega\subset\rn$ be a domain and let $n-1<q<\fz$. Then $\Omega$ is $q$-capacitory dense at every point of the boundary. A planar domain $\boz\subset\rr^2$ is also $1$-capacitory dense at every point of the boundary.
\end{thm}

\begin{proof}
Fix $x\in\partial\boz$. Given $0<t<\diam(\boz)/3,$ we may pick points $z\in B(x,t/2)\cap \boz$ and $y\in \boz\setminus B(x,3t).$ Since $\boz$ is open and connected, we find a curve $\gamma$ that joins $z$ to $y$ in $\boz.$ This curve gives us connected sets $E_t\subset \boz\cap B(x, t)$ and
$F_t\subset \boz \cap A(x;2t,3t)$ with 
\[\frac{t}{2}\leq\diam E_t\]
and 
\[\diam(F_{t})\geq t.\] Hence, by Lemma \ref{lem2.1}, for every $n-1<q<\fz$, we have 
\[ Cap_q( \boz\cap B(x,t),\boz\cap  A(x; 2t, 3t); B(x, 4t))\geq Cap_q(E_t, F_t; B(x, 4t))\geq Ct^{n-q}\]
for some positive constant $C$ independent of $x$ and $t$. By (\ref{equa:cap}) 
we conclude that
\begin{equation}
\limsup_{t\to0^+}\frac{Cap_q\lf(\boz\cap B(x, t), \boz\cap A(x; 2t,3t); B(x, 4t)\r)}{Cap_q\lf(B(x, t), A(x; 2t, 3t); B(x,4t)\r)}=\delta_x>0.\nonumber
\end{equation}
Consequently, the domain $\boz$ is $q$-capacitory dense at the point $x\in\partial\boz$.

Finally, let us assume that $n=2$ and $q=1$. Similarly as above, by Lemma \ref{lem2.1}, we have
\[Cap_1\lf(\boz\cap B(x,t), \boz\cap A(x; 2t, 3t); B(x, 4t)\r)\geq Cap_1\lf(E_t, F_t; B(x, 4t)\r)\geq Ct\]
Since
\[Cap_1\lf(B(x, t), A(x; 2t, 3t); B(x, 4t)\r)\sim_c t\]
 for some positive constant $c$ independent of $x$ and $t$, we have
 \begin{equation} 
 \limsup_{t\to 0} \frac{Cap_1(\boz\cap B(x, t),\boz\cap A(x; 2t, 3t); B(x, 4t))}{Cap_1(B(x, t), A(x; 2t, 3t); B(x, 4t))}=\delta_x>0.\nonumber
 \end{equation}
 Consequently, the domain $\boz\subset\rr^2$ is $1$-capacitory dense at the point $x\in\partial\boz$.
\end{proof}

\begin{remark}\label{vahva}
We actually proved that
\begin{equation} \label{vahvempi}
Cap_q( \boz\cap B(x,t),\boz\cap  A(x; 2t, 3t); B(x, 4t))\geq Cap_q(E_t, F_t; B(x, 4t))\geq C_qt^{n-q}
\end{equation}
whenever $x\in \partial \boz,$ $0<t<\frac 1 4 \diam(\boz)$ and $q>n-1$ (also for $q=1$ in the plane).
\end{remark}

\section{Proofs of Corollary \ref{thm:capa} and Theorems \ref{thm:poicade}, \ref{thm:point},  \ref{thm:high} and \ref{thm:strong}}

\begin{proof}[Proof of Theorem \ref{thm:capa}]
Let $\boz\subset\rn$ be a Sobolev $(p, q)$-extension domain with $1\leq q<p<\fz$, and let $E:W^{1,p}(\boz)\to W^{1,q}(\rn)$ be the corresponding homogeneous bounded Sobolev extension operator from Lemma \ref{lem:homo}. Define the associated set function $\Phi$ by (\ref{setfunc11}). Let $x\in\overline{\boz}$ and $0<r<\min\{1, \frac{1}{4}\diam(\boz)\}$ be fixed. Then we define a function $u\in W^{1,p}(\boz)\cap C(\boz)$ by setting
\begin{equation}\label{equa:testu'2}
u(y)= \begin{cases} 1 & \textnormal{in } B(x, \frac{r}{4})\cap\boz, \\
\frac{-4}{r}|y-x|+2 & \textnormal{in }\lf(B(x, \frac{r}{2})\setminus B(x,\frac{r}{4})\r)\cap\boz,\\
 0 & \textnormal{in } \boz\setminus B(x,\frac{r}{2}) \, . \end{cases}
\end{equation}
We have 
\begin{equation}\label{eq:SNu}
\lf(\int_{B(x,r)\cap\boz}|u(y)|^pdy+\int_{B(x,r)\cap\boz}|\nabla u(y)|^pdy\r)^{\frac{1}{p}}\leq \frac{C}{r}|B(x, r)\cap\boz|^{\frac{1}{p}}.
\end{equation}
Because $u\in C(\boz\cap B(x, r))$ with $u\equiv 0$ on $\boz\cap\partial B(x,r)$, we conclude that $u\in W^p_0(B(x, r), \boz)$. By Corollary \ref{cor:sofun1}, we have 
\begin{equation}\label{eq:normin}
\lf(\int_{B(x, r)}|\nabla E(u)(y)|^qdy\r)^{\frac{1}{q}}\leq C(\Phi(B(x, r)))^{\frac{1}{k}}\lf(\int_{B(x, r)\cap\boz}|u(y)|^p+|\nabla u(y)|^pdx\r)^{\frac{1}{p}}
\end{equation}
with $\frac{1}{k}=\frac{1}{q}-\frac{1}{p}$. By (\ref{equa:testu'2}) and the density of continuous functions in $W^{1,q}(B(x, r))$, it is easy to check there exists a sequence in $\mathcal W_q\lf(B\lf(x, \frac{r}{4}\r)\cap\boz, \boz\cap A\lf(x; \frac{r}{2}, \frac{3r}{4}\r); B(x, r)\r)$ that converges to $E(u)$ both almost everywhere and in the Sobolev norm. Hence
\begin{equation}\label{eq:caplow}
\int_{B(x, r)}|\nabla E(u)(x)|^qdx\geq Cap_q\lf(B\lf(x, \frac{r}{4}\r)\cap\boz, \boz\cap A\lf(x; \frac{r}{2}, \frac{3r}{4}\r); B(x, r)\r).
\end{equation}
By combining inequalities (\ref{eq:SNu}), (\ref{eq:normin}) and (\ref{eq:caplow}), we obtain the inequality 
$$C\Phi(B(x, r))^{p-q}|B(x, r)\cap\boz|^q\geq r^{pq}Cap_q\lf(\boz\cap B\lf(x, \frac{r}{4}\r), \boz\cap A\lf(x; \frac{r}{2}, \frac{3r}{4}\r); B(x, r)\r)^p.$$
Our claim follows for the set function $\hat \Phi:=c\Phi,$ where $c=C^{1/(p-q)}.$
\end{proof}

\begin{proof}[Proof of Theorem \ref{thm:point}]
Suppose that $\Omega$ is $q$-fat at almost every $x\in\partial\boz$. By the Lebesgue density theorem and Lemma~\ref{lem:upfinite}, there exists a subset $V\subset\overline\boz$ with $|V|=|\overline\boz|$ such that every $x\in V$ is a Lebesgue point of $\overline\boz$ and $\boz$ is $q$-fat at every $x\in V$. Fix $x\in V$. Let $\epsilon>0$ be sufficiently small such that $1-\epsilon\geq \frac{1}{2^{n-1}}.$ Since $x\in V$ is a Lebesgue point of $\overline\boz$, there exists $0<r_x<1$ such that for every $0<r<r_x$, we have 
\begin{equation}\label{equa:epsilon}
|B(x,r)\cap\overline\boz|\geq (1-\epsilon)|B(x,r)|\geq\frac{1}{2^{n-1}}|B(x, r)|.
\end{equation}
Let $r\in(0, r_x)$ be fixed. Since $|\partial B(x, s)|=0$ for every $0<s<r$, we have
\begin{equation}\label{eq:volume1}
\left|B\lf(x, \frac{r}{4}\r)\cap\overline\boz\r|\geq\frac{1}{2^{n-1}}\lf|B\lf(x, \frac{r}{4}\r)\r|\geq\frac{1}{2^{3n-1}}|B(x, r)|
\end{equation}
and
\begin{equation}\label{eq:volume2}
\lf|\lf(B(x, r)\setminus B\lf(x, \frac{r}{2}\r)\r)\cap\overline\boz\r|\geq \lf|B(x, r)\cap\overline\boz\r|-\lf|B\lf(x, \frac{r}{2}\r)\r|\geq\frac{1}{2^{n}}|B(x, r)|.
\end{equation}

Let $u$ be defined by (\ref{equa:testu'2}). Let $\Phi$ be the set function from (\ref{setfunc11}). 
Then $u\in W^p_0(B(x, r), \boz)$. By \eqref{eq:normin} and \eqref{eq:SNu}, we have $E(u)\in W^{1,q}(B(x,r))$ with
\begin{equation}\label{equa:equa0'1}
\lf(\int_{B(x, r)}|\nabla E(u)(y)|^qdy\r)^{\frac{1}{q}}\leq C\lf(\Phi(B(x, r))\r)^{\frac{1}{\kappa}}
\frac C r |B(x, r)\cap\boz|^{\frac{1}{p}}
\end{equation}
with $\frac{1}{\kappa}=\frac{1}{q}-\frac{1}{p}$. Since $\boz$ is $q$-fat at every $y\in V$, Lemma \ref{lem:u=1} implies that $E(u)(y)=0$ for almost every $y\in (B(x, r)\setminus B(x, \frac{r}{2}))\cap V$ and $E(u)(y)=1$ for almost every $y\in B(x, \frac{r}{4})\cap V$. Since $|V|=|\overline\boz|$, $E(u)(y)=1$ for almost every $y\in B(x, \frac{r}{4})\cap\overline{\boz}$ and $E(u)(y)=0$ for almost every $y\in (B(x, r)\setminus B(x, \frac{r}{2}))\cap\overline\boz$. 

By the Poincar\'e inequality on balls, we have 
\begin{equation}\label{equa:lower1}
Cr^q\int_{B(x, r)}|\nabla E(u)(y)|^qdy\geq \int_{B(x,r)}|E(u)(y)-E(u)_{B(x, r)}|^{q}dy.
\end{equation}
If $E(u)_{B(x, r)}\geq \frac{1}{2}$, since $E(u)(y)=0$ for almost every $y\in (B(x, r)\setminus B(x, \frac{r}{2}))\cap\overline\boz$, we conclude from (\ref{eq:volume2}) that
\begin{multline}
\int_{B(x, r)}|E(u)(y)-E(u)_{B(x, r)}|^{q}dy\geq \lf(\frac{1}{2}\r)^{q}\Big|\lf(B(x,r)\setminus B\lf(x, \frac{r}{2}\r)\r)\cap\overline\boz\Big|
                                               \geq C|B(x,r)|.\nonumber
\end{multline} 
In the case $E(u)_{B(x, r)}<\frac{1}{2}$, since $E(u)(y)=1$ for almost every $y\in B(x, \frac{r}{4})\cap\overline\boz$, we conclude from (\ref{eq:volume1}) that 
$$
\int_{B(x,r)}|E(u)(y)-E(u)_{B(x, r)}|^{q}dy \geq \lf(\frac{1}{2}\r)^{q}\Big|B\lf(x, \frac{r}{4}\r)\cap\overline\boz\Big|\nonumber
                                               \geq C|B(x,r)|.\nonumber
$$
In conclusion, we always have
\begin{equation}\label{equa:equation11}
\int_{B(x,r)}|E(u)(y)-E(u)_{B(x, r)}|^{q}dy\geq C|B(x,r)|.
\end{equation} 
By combining inequalities  (\ref{equa:equa0'1}), (\ref{equa:lower1}) and (\ref{equa:equation11}), we obtain the inequality 
$$
\Phi(B(x, r))^{p-q}|B(x, r)\cap\boz|^q\geq C|B(x, r)|^p.
$$
The desired inequality follows from this by replacing $\Phi$ with $c\Phi$ for a suitable $c.$
\end{proof}

\begin{proof}[Proof of Theorem \ref{thm:high}]
Let us assume that $|\partial\boz|>0$. Then, by the Lebesgue density theorem (see, for example \cite{Stein})
 and Theorem \ref{thm:point}, there exists a subset $V\subset\partial\boz$ with $|V|=|\partial\boz|>0$ such that every point $x\in V$ is a Lebesgue point of $\partial\boz$, $\overline D\Phi(x)<\infty$ and
\begin{equation}\label{equa1:pq}
\Phi(B(x,r))^{p-q}|B(x,r)\cap\boz|^q\geq C|B(x,r)|^p
\end{equation}
holds for every $x\in V$ and each $0<r<r_x.$ Fix $x\in V$. Then by  inequality (\ref{equa1:pq}), we have 
$$
|B(x, r)\cap\partial\boz|\leq |B(x, r)|-|B(x, r)\cap\boz|\leq |B(x, r)|-C\frac{|B(x,r)|^{\frac{p}{q}}}{\Phi(B(x,r))^{\frac{p-q}{q}}},
$$
for every $0<r<r_x$. Hence, by Lemma~\ref{lem:upfinite}, we obtain
\begin{multline}
\nonumber
\limsup\limits_{r\to 0^+}\frac{|B(x, r)\cap\partial\boz|}{|B(x,r)|}\leq \limsup\limits_{r\to 0^+}\frac{|B(x, r)|-|B(x, r)\cap\boz|}{|B(x,r)|}\\
\leq \limsup\limits_{r\to 0^+}\frac{|B(x, r)|}{|B(x,r)|}-C\liminf\limits_{r\to 0^+}\frac{|B(x,r)|^{\frac{p-q}{q}}}{\Phi(B(x,r))^{\frac{p-q}{q}}}
=1-C\overline{D}\Phi(x)^{\frac{q-p}{p}}<1.
\end{multline}

This contradicts the assumption that $x\in V$ is a Lebesgue point of $\partial\boz$. We conclude that $|\partial\boz|=0$. 
\end{proof}

\begin{proof}[Proof of Theorem \ref{thm:strong}]
Let $\boz\subset\rn$ be a Sobolev $(p, q)$-extension domain with $1\leq q\leq p<\fz$. First, if $|\partial\boz|=0$, by Definition \ref{defn:strong}, every bounded extension operator 
$E: W^{1,p}(\boz)\to W^{1,q}(\rn)$ is a strong bounded extension operator. 

Conversely, let us assume that there exists a strong bounded extension operator $E_s:W^{1,p}(\boz)\to W^{1,q}(\rn)$. Fix a function $u\in W^{1,p}(\boz)$ as in (\ref{equa:testu'2}). Since $E_s$ is a strong bounded extension operator, we have $E_s(u)(y)=1$ for almost every $y\in B(x, \frac{r}{4})\cap\overline\boz$ and $E_s(u)(y)=0$ for almost every $y\in (B(x, r)\setminus B(x, \frac{r}{2}))\cap\overline\boz$. Hence, similarly to the proof of Theorem \ref{thm:point}, we obtain the point-wise density inequality (\ref{equa:meade}) for almost every $x\in\overline\boz$. Finally, by making use of Lebesgue density theorem and repeating the proof of Theorem \ref{thm:high}, we conclude that $|\partial\boz|=0$.
\end{proof}

\begin{proof}[Proof of Corollary \ref{thm:poicade}]
The claim follows by combining Proposition \ref{le:stro} with Theorem \ref{thm:point}.
\end{proof}

\section{Proof of Theorem \ref{lem:cacusp}}\label{sc:sharp}

In this section, we prove Theorem \ref{lem:cacusp} that gives the sharp capacity estimate for outward cusp domains. 
After this, we use doubling order outward cusp domains to construct examples towards the sharpness of inequality (\ref{eq:capden}).

\begin{proof}[Proof of Theorem \ref{lem:cacusp}]
For arbitrary $1\leq p<\fz$, we always have
\begin{multline} 
\mathcal W_p\lf(\boz^n_w\cap B\lf(0, \frac{r}{4}\r), A\lf(0; \frac{r}{2}, \frac{3r}{4}\r); B(0, r)\r)\nonumber\\
\subset\mathcal W_p\lf(\boz_w^n\cap B\lf(0, \frac{r}{4}\r), \boz_w^n\cap A\lf(x; \frac{r}{2}, \frac{3r}{4}\r); B(0, r)\r).
\end{multline}
Hence,
\begin{multline}
Cap_p\lf(\boz_w^n\cap B\lf(0, \frac{r}{4}\r), A\lf(0; \frac{r}{2}, \frac{3r}{4}\r); B(0, r)\r)\nonumber\\
\geq Cap_p\lf(\boz_w^n\cap B\lf(0, \frac{r}{4}\r), \boz_w^n\cap A\lf(0; \frac{r}{2}, \frac{3r}{4}\r); B(0, r)\r).
\end{multline}

We divide the argument for the remaining inequalities into three cases.

\textbf{The case $n-1<p<\fz$:} By Lemma \ref{lem2.1}, we have 
\[Cap_p\lf(\boz^n_w\cap B\lf(0,\frac{r}{4}\r),\boz^n_{w}\cap A\lf(0; \frac{r}{2}, \frac{3r}{4}\r); B(0, r)\r)\geq Cr^{n-p}.\] 
We define a test function $v$ on $B(0, r)$ by setting
\begin{equation}\label{eq:test1}
v(z):=\begin{cases}
1 &\textnormal{if }\ |z|<\frac{r}{4}\\
\frac{-4}{r}|z|+2 &\textnormal{if }\ \frac{r}{4}\leq |z|\leq\frac{r}{2}\\
0 &\textnormal{if }\ \frac{r}{2}<|z|<r\, \end{cases}.
\end{equation}
Since $v\in\mathcal W_p\lf(\boz_w^n\cap B\lf(0, \frac{r}{4}\r), A\lf(0; \frac{r}{2}, \frac{3r}{4}\r); B(0, r) \r)$, we have 
\[Cap_p\lf(\boz_w^n\cap B\lf(0,\frac{r}{4}\r), A\lf(0; \frac{r}{2}, \frac{3r}{4}\r); B(0, r)\r)\leq \int_{B(0, r)}|\nabla v(z)|^pdz\leq Cr^{n-p}.\]

\textbf{The case $1\leq p<n-1$:} Given $\frac{r}{5}<\rho<\frac{r}{4}$, we define an $(n-1)$-dimensional sphere $S_\rho$ by 
\[S_\rho:=\lf\{z\in\rn: d\lf(z, \lf(\frac{3r}{8}, 0, \cdots, 0\r)\r)=\rho\r\}.\]
We set
\[S^+_\rho:=\lf\{z=(t, x_1,x_2,\cdots,x_{n-1})\in S_\rho:x_{n-1}>0 \r\}\]
and let $A^+_1(\rho):= S^+_\rho\cap\lf(B\lf(0, \frac{r}{4}\r)\cap\boz^n_{w}\r)$ and $A^+_0(\rho):=S^+_\rho\cap\lf(\boz^n_{w}\setminus B\lf(0, \frac{r}{2}\r)\r)$. Since 
$w$ is doubling and $$\lim_{r\to0^+}w'(r)=0,$$ 
we have
$$\mathcal H^{n-1}(A^+_0(\rho))\sim_c (w(r))^{n-1}\ {\rm and}\ \mathcal H^{n-1}(A^+_1(\rho))\sim_c (w(r))^{n-1}$$
for every $\rho\in(\frac{r}{5}, \frac{r}{4})$. The implicit constants are independent of $r$ and $\rho$. There exists a bi-Lipschitz homeomorphism from $S^+_\rho$ to the $(n-1)$-dimensional disk $B^{n-1}(0, \rho)$ with a bi-Lipschitz constant independent of $\rho$, for example, see \cite[Lemma 2.19]{HandK}. Hence, for each $v\in\mathcal W_p\lf(B\lf(0, \frac{r}{4}\r)\cap\boz^n_{w}, \boz^n_{w}\cap A\lf(0; \frac{r}{2}, \frac{3r}{4}\r); B(0, r) \r)$, by the Sobolev-Poincar\'e inequality on balls \cite[Theorem 4.9]{Evans}, for almost every $\rho\in (\frac{r}{5}, \frac{r}{4})$, we have 
\begin{equation}\label{eq:poinB}
\lf(\bint_{S^+_\rho}|v(z)-v_{S^+_\rho}|^{p^\star} dz\r)^{\frac{1}{p^\star}}\leq Cr\lf(\bint_{S^+_\rho}|\nabla v(z)|^pdz\r)^{\frac{1}{p}}
\end{equation}
with $p^\star=\frac{(n-1)p}{n-1-p}$. Assuming $v_{S^+_\rho}\leq\frac{1}{2}$, we have 
\begin{eqnarray}
(w(r))^{n-1-p}&\leq&C\lf(\int_{A^+_1(\rho)}|v(z)-v_{S^+_\rho}|^{p^\star}dz\r)^{\frac{p}{p^\star}}\nonumber\\
               &\leq&C\lf(\int_{S^+_\rho}|v(z)-v_{S^+_\rho}|^{p^\star}dz\r)^{\frac{p}{p^\star}}\leq C\int_{S^+_\rho}|\nabla v(z)|^pdz.\nonumber
\end{eqnarray}
If $v_{S^+_\rho}>\frac{1}{2}$, we simply replace $A^+_1(\rho)$ by $A^+_0(\rho)$ in the inequality above. Hence, for almost every $\rho\in (\frac{r}{5}, \frac{r}{4})$, we have 
\[(w(r))^{n-1-p}\leq C\int_{S^+_\rho}|\nabla v(z)|^pdz.\]  
By integrating over $\rho\in(\frac{r}{5}, \frac{r}{4})$, we obtain 
\[(w(r))^{n-1-p}r\leq C\int_{\frac{r}{5}}^{\frac{r}{4}}\int_{S^+_\rho}|\nabla v(z)|^pdzd\rho\leq C\int_{\rn}|\nabla v(z)|^pdz.\]
Since $v\in\mathcal W_p\lf(B\lf(0, \frac{r}{4}\r)\cap\boz^n_{w}, \boz^n_{w}\cap A\lf(0; \frac{r}{2}, \frac{3r}{4}\r); B(0, r) \r)$ is arbitrary, we conclude that 
\[
Cap_p\lf(B\lf(0, \frac{r}{4}\r)\cap\boz^n_{w}, \boz^n_{w}\cap A\lf(0; \frac{r}{2}, \frac{3r}{4}\r); B(0, r) \r)
\geq c_2(w(r))^{n-1-p}r.
\]

Towards the other direction of the inequality, we construct a suitable test function. We define a cut-off function $F_1$ by setting
\[F_1(z)=F_1(t, x):=\begin{cases}
1 &\textnormal{if }\ |x|<w(\frac{r}{4})\\
\frac{-|x|}{w(\frac{r}{2})-w(\frac{r}{4})}+\frac{w(\frac{r}{2})}{w(\frac{r}{2})-w(\frac{r}{4})} &\textnormal{if }\ w(\frac{r}{4})\leq |x|\leq w(\frac{r}{2})\\
0 &\textnormal{if } |x|>w(\frac{r}{2})\, \end{cases}.\]
Then we define our test function $v_1\in\mathcal W_p\lf(B\lf(0, \frac{r}{4}\r)\cap\boz^n_{w}, A\lf(0; \frac{r}{2}, \frac{3r}{4}\r); B(0, r) \r)$ by $v_1(z):=v(z)F_1(z),$ where $v$ is defined in \eqref{eq:test1}. 
Since $w'$ is increasing on $(0, \fz)$ and $w$ is doubling, we have 
$$w\lf(\frac{r}{4}\r)\leq w\lf(\frac{r}{2}\r)-w\lf(\frac{r}{4}\r)\leq w\lf(\frac{r}{2}\r)\leq w(r)\leq Cw\lf(\frac{r}{4}\r).$$  
Hence, a simple computation shows that
\[|\nabla v_1(z)|\leq\begin{cases}
\frac{C}{w(r)} &\textnormal{if }\ |t|<\frac{r}{2}\ {\rm and}\ |x|<w(r)\\

0 &\textnormal{otherwise }\, \end{cases}.\]
This implies 
\[Cap_p\lf(B\lf(0, \frac{r}{4}\r)\cap\boz^n_{w}, A\lf(0; \frac{r}{2}, \frac{3r}{4}\r); B(0, r) \r)\leq\int_{B(0, r)}|\nabla v_1(z)|^pdz\leq Cr(w(r))^{n-1-p}.\]

\textbf{The case $p=n-1$:} Let $z_1:=(-\rho, 0, \cdots, 0)$ and $z_2:=(\rho, 0,\cdots, 0)$ be a pair of antipodal points on the $(n-2)$-dimensional sphere $\partial B^{n-1}(0, \rho)$. Denote $\tilde A^+_1(\rho):= B^{n-1}(z_1, w(\rho))\cap B^{n-1}(0, \rho)$ and $\tilde A^+_0(\rho):=B^{n-1}(z_2, w(\rho))\cap B^{n-1}(0, \rho)$.  For every $\rho\in(0, \frac{1}{4})$, there exists a bi-Lipschitz homeomorphism $H_\rho: S^+_\rho\to B^{n-1}(0, \rho)$ with $\tilde A^+_1(\rho)=H_\rho(A^+_1(\rho))$, $\tilde A^+_0(\rho)=H_\rho(A^+_0(\rho))$, with bi-Lipschitz constant independent of $\rho$. Let 
$$\{0\}\times\rr^{n-2}:=\{x=(0, x_2, x_3,\cdots, x_{n-1}):x_i\in\rr\ {\rm for}\ i=2, 3,\cdots, n-1\}.$$
For $z\in \{0\}\times\rr^{n-2}\cap B^{n-1}(0, \rho)$, we define $L_{z_1}^z$ to be the line segment with endpoints $z_1, z$ and $L_{z_2}^z$ to be the line segment with endpoints $z_2, z$. We also define $S_{z_1}^z:=L_{z_1}^z\setminus B^{n-1}(z_1, w(\rho))$ and $S_{z_2}^z:=L_{z_2}^z\setminus B^{n-1}(z_2, w(\rho))$. Fix a test function 
$$\hat v\in\mathcal W_{n-1}\lf(B\lf(0, \frac{r}{4}\r)\cap\boz^n_{w}, \boz^n_{w}\cap A\lf(0; \frac{r}{2}, \frac{3r}{4}\r); B(0, r) \r).$$
 The function $\tilde v_\rho$ defined by $\tilde v_\rho:=\hat v\circ H_\rho^{-1}$, is continuous on $B^{n-1}(0, \rho)$ with $\tilde v_\rho\big|_{\tilde A^+_1(\rho)}\geq 1$ and $\tilde v_\rho\big|_{\tilde A^+_0(\rho)}\leq 0$. By the Fubini theorem, for almost every $\rho\in(\frac{r}{5}, \frac{r}{4})$, $\tilde v_\rho \in W^{1,n-1}(B^{n-1}(0, \rho))\cap C(B^{n-1}(0, \rho))$. Let us fix such a $\rho\in(\frac{r}{5}, \frac{r}{4})$. Then for $\mathcal H^{n-2}$-a.e. $z\in\{0\}\times\rr^{n-2}\cap B^{n-1}(0, \rho)$, by the fundamental theorem of calculus, we have either 
\[\frac{1}{2}\leq\int_{S_{z_1}^z}|\nabla\tilde v_\rho(x)|dx\ \ \ 
{\rm or}\ \ \  
\frac{1}{2}\leq\int_{S_{z_2}^z}|\nabla\tilde v_\rho(x)|dx.\] 
Then the H\"older inequality implies either
\[\lf(\int_{S_{z_1}^z}|x-z_1|^{-1}dx\r)^{2-n}\leq C\int_{S_{z_1}^z}|\nabla\tilde v_\rho(x)|^{n-1}|x-z_1|^{n-2}dx\]
or
\[\lf(\int_{S_{z_2}^z}|x-z_2|^{-1}dx\r)^{2-n}\leq C\int_{S_{z_2}^z}|\nabla\tilde v_\rho(x)|^{n-1}|x-z_2|^{n-2}dx.\]
Hence, we have either
\[\frac{1}{\log^{n-2}\frac{r}{w(r)}}\leq C\int_{B^{n-1}(z_1, \sqrt{2}\rho)\cap B^{n-1}(0, \rho)}|\nabla\tilde v_\rho(x)|^{n-1}dx\]
or
\[\frac{1}{\log^{n-2}\frac{r}{w(r)}}\leq C\int_{B^{n-1}(z_2, \sqrt{2}\rho)\cap B^{n-1}(0, \rho)}|\nabla\tilde v_\rho(x)|^{n-1}dx.\]
In conclusion, for every $\rho\in\lf(\frac{r}{5}, \frac{r}{4}\r)$ with $\tilde v_\rho\in W^{1, n-1}(B^{n-1}(0, \rho))$, we have 
\[\frac{1}{\log^{n-2}\frac{r}{w(r)}}\leq C\int_{B^{n-1}(0, \rho)}|\nabla\tilde v_\rho(x)|^{n-1}dx.\]
Since, for every $\rho\in\lf(\frac{r}{5}, \frac{r}{4}\r)$, $H_\rho: S^+_\rho\to B^{n-1}(0, \rho)$ is bi-Lipschitz with bi-Lipschitz constant independent of $\rho$, we have 
\[\frac{1}{\log^{n-2}\frac{r}{w(r)}}\leq C\int_{S^+_\rho}|\nabla\hat v(z)|^{n-1}dz.\]
By integrating over $\rho\in(\frac{r}{5},\frac{r}{4})$, we obtain 
\[\frac{r}{\log^{n-2}\frac{r}{w(r)}}\leq C\int_{B(0, r)}|\nabla\hat v(z)|^{n-1}dz.\]
Since $\hat v\in\mathcal W_{n-1}\lf(B\lf(0, \frac{r}{4}\r)\cap\boz^n_{w}, \boz^n_{w}\cap A\lf(0; \frac{r}{2}, \frac{3r}{4}\r); B(0, r) \r)$ is arbitrary, we conclude that
\[Cap_{n-1}\lf(\boz_w^n\cap B\lf(0, \frac{r}{4}\r), \boz^n_{w}\cap A\lf(0; \frac{r}{2}, \frac{3r}{4}\r); B(0, r) \r)\geq C\frac{r}{\log^{n-2}\frac{r}{w(r)}}.\]

Towards the opposite direction of this inequality, we construct a suitable test function. We define a cut-off function $F_2$ by setting
\[F_2(z)=F_2(t, x):=\begin{cases}
1 &\textnormal{if }\ |x|<w(\frac{r}{4})\\
\frac{\log\frac{4|x|}{r}}{\log\frac{4w(\frac{r}{4})}{r}} &\textnormal{if }\ w(\frac{r}{4})\leq |x|\leq\frac{r}{4}\\
0 &\textnormal{if } |x|>\frac{r}{4}\, \end{cases}.\]
Then we define our test function $v_2\in\mathcal W_{n-1}\lf(B\lf(0, \frac{r}{4}\r)\cap\boz^n_{u}, A\lf(0; \frac{r}{2}, \frac{3r}{4}\r); B(0, r) \r)$ by 
\[v_2(z):=\begin{cases}
F_2(z) &\textnormal{if }\ |z|<\frac{r}{4}\\
F_2(z)\frac{\log\frac{2|z|}{r}}{\log\frac{1}{2}} &\textnormal{if }\ \frac{r}{4}\leq |z|\leq\frac{r}{2}\\
0 &\textnormal{if } |z|>\frac{r}{2}\, \end{cases}.\] 
Since $w$ is doubling, a simple computation shows that
\[|\nabla v_2(z)|\leq\begin{cases}
\frac{C}{|x|\log\frac{r}{w(r)}} &\textnormal{if }\ |t|<\frac{r}{2}\ {\rm and}\ w(\frac{r}{4})<|x|<\frac{r}{4}\\

0 &\textnormal{elsewhere }\, \end{cases}.\]
Hence, 
\[Cap_{n-1}\lf(B\lf(0, \frac{r}{4}\r)\cap\boz^n_{w}, \boz^n_{w}\cap A\lf(0; \frac{r}{2}, \frac{3r}{4}\r); B(0, r) \r)\leq\int_{\rn}|\nabla v_2(z)|^{n-1}dz\leq \frac{Cr}{\log^{n-2}\frac{r}{w(r)}}.\]

By combining the three cases above, we obtain the missing inequalities.
\end{proof}

We proceed to show the sharpness of the inequality (\ref{eq:capden}). We need the following lemma.

\begin{lem}\label{lem:PhiHau}
Let $0\leq \lambda<n$ and $\Phi$ be a non-negative, bounded, monotone and countably additive  
set function defined on open sets. Define
$$
E=\lf\{x\in\overline\boz:\limsup_{r\to0}\frac{\Phi(B(x, r))}{r^\lambda}=\infty\r\}.
$$
Then
$$
\mathcal H^\lambda(E)=0.
$$
\end{lem}
\begin{proof}
For each $x\in E$ and every $\delta>0$, there exists $0<r_x<\delta$ such that
$$
\delta \Phi(B(x, r_x))>r_x^\lambda.
$$
Define 
$$
\mathcal F:=\lf\{B(x, r_x):x\in E\r\}.
$$
By the classical Vitali covering theorem,  there exists an at most countable subclass of pairwise disjoint balls $\{B_i\}_{i=1}^{\fz}$ in $\mathcal F$ such that 
$$
E\subset \bigcup_{i=1}^\fz5B_i.
$$
Hence, writing $r_i$ for the radius of $B_i$, we have
\begin{eqnarray}
\mathcal H^\lambda_{10\delta}(E)&\leq& C\sum_{i=1}^\fz(5r_i)^\lambda\leq C\delta\sum_{i=1}^\fz\Phi(B_i)\nonumber\\
                                                                              &\leq&C\delta\Phi\lf(\bigcup_{i=1}^\fz B_i\r)\leq C\delta\Phi(\rn).\nonumber
\end{eqnarray}
The claim follows by letting $\delta$ tend to zero.
\end{proof}

{\begin{proof}[Sharpness of (\ref{eq:capden})]
We use outward cusp domains  to construct Sobolev  extension domains that show the sharpness of (\ref{eq:capden}).  Given $s\in (1, \fz)$ and $\alpha>\frac {s-1}{s},$ let $\omega(t)=t^s\log^{\alpha}(\frac e t),$ and 
consider the outward cusp domain 
$\boz^n_{t^s\log^{\alpha}(\frac{e}{t})}:=\boz_\omega^n\subset\rn.$ By results due to Maz'ya and Poborchi in \cite{Mazya1, Mazya2, Mazya, Mazya3}, we have the following results. For $n\geq 3$, 
$\boz^n_{t^s\log^\alpha(\frac{e}{t})}$ is a Sobolev $(p, q)$-extension domain for
\begin{equation}\label{equa:pq0}
\begin{cases} 1\leq q\leq\frac{(1+(n-1)s)p}{1+(n-1)s+(s-1)p} & \textnormal{if } \frac{1+(n-1)s}{2+(n-2)s}\leq p\leq\frac{(n-1)+(n-1)^2s}{n}, \\
 1\leq q\leq\frac{np}{1+(n-1)s} & \textnormal{if } \frac{(n-1)+(n-1)^2s}{n}\leq p<\fz \, . \end{cases}
\end{equation}
For $n=2$, $\boz^n_{t^s\log^\alpha(\frac{e}{t})}$ is a Sobolev $(p, q)$-extension domain for $\frac{1+s}{2}\le p<\fz$ and $1\leq q\le \frac{2p}{1+s}$. 

Clearly, there exists a constant $C>1$ such that for every $0<r<1$, we have 
\begin{equation}\label{eq:EQU1'}
\frac{1}{C}r^{1+(n-1)s}\log^{\alpha(n-1)}\lf(\frac{e}{r}\r)\leq|B(0,r)\cap\boz^n_{t^s\log^\alpha(\frac e t)}|\leq Cr^{1+(n-1)s}\log^{\alpha(n-1)}\lf(\frac{e}{r}\r).
\end{equation}
Furthermore, \eqref{eq:cacuspH} gives us a lower bound for the capacitory term in \eqref{eq:capden}
in terms of $r,q,s,n$ and $\log^\alpha\lf(\frac{e}{r}\r).$

By comparing the capacity estimate, (\ref{eq:EQU1'}) and  (\ref{eq:cacuspH})  for the values of $q,p$ given by \eqref{equa:pq0} we see that \eqref{eq:capden} cannot hold for a bounded set function $\Phi$ for better
exponents than the given ones.



Let us also analyze the additivity of $\Phi.$ Fix $n\ge 3.$
Let $k\in\{1, 2,\cdots, n-2\},$ $s\in(1, \fz)$ and $\alpha>\frac{s-1}{k+1}$ be fixed. We define a domain $G^k_n(s,\alpha)\subset\rn$ by setting 
$$
G^k_n(s,\alpha):=\boz^{k+1}_{t^s\log^\alpha(\frac{e}{t})}\times\rr^{n-k-1}.
$$
Since $G^k_n(s,\alpha)$ is the product of $\boz^{k+1}_{t^s\log^{\alpha}(\frac{e}{t})}$ and $\rr^{n-k-1}$, by the extension results in \cite{Mazya1, Mazya2, Mazya, Mazya3} and product results in \cite{KZiumj, Zhu}, we obtain the following conclusions. For $k\geq 2$, $G^k_n(s,\alpha)$ is a Sobolev $(p, q)$-extension domain for 
\begin{equation}\label{equa:pq}
\begin{cases} 1\leq q\leq\frac{(1+ks)p}{1+ks+(s-1)p} & \textnormal{if } \frac{1+ks}{2+(k-1)s}\leq p\leq\frac{k +k^2s}{k+1}, \\
 1\leq q\leq\frac{(k+1)p}{1+ks} & \textnormal{if } \frac{k+k^2s}{k+1}\leq p<\fz \, ,\end{cases}
\end{equation}
and $G^1_n(s,\alpha)$ is a Sobolev $(p, q)$-extension domain for $\frac{1+s}{2}\leq p<\fz$ and $1\leq q\leq\frac{2p}{1+s}$.

Clearly, there exists a constant $C>1$ such that, for every $x\in\{0\}\times\rr^{n-k-1}$ and each $0<r<1$, we have 
\begin{equation}\label{eq:EQU1}
\frac{1}{C}r^{n+ks-k}\log^{\alpha k}\lf(\frac{e}{r}\r)\leq|B(x, r)\cap G^k_n(s,\alpha)|\leq C r^{n+ks-k}\log^{\alpha k}\lf(\frac{e}{r}\r).
\end{equation}
Moreover, Fubini theorem, Theorem \ref{lem:cacusp} and Lemma \ref{lem2.1} give with some work the estimates
\begin{multline}\label{eq:EQU2}
Cap_q\lf(G^k_n(s,\alpha)\cap B\lf(x, \frac{r}{4}\r), G^k_n(s,\alpha)\cap A\lf(x; \frac{r}{2}, \frac{3r}{4}\r); B(x, r)\r)\\
\geq\begin{cases}
c_1 r^{n-q} &\textnormal{if }\ k<q<\fz\\
\frac{c_2r^{n-k}}{\log^{k}\frac{e}{r}} &\textnormal{if }\ q=k\\
c_3r^{(k-q)s+n-k}\log^{\alpha(k-q)}\lf(\frac{e}{r}\r) &\textnormal{if }\ 1\leq q<k\, \end{cases}
\end{multline}
and, for $k=1$,  
\begin{equation}\label{eq:EQU2'}
Cap_q\lf(G^1_n(s,\alpha)\cap B\lf(x, \frac{r}{4}\r), G^1_n(s,\alpha)\cap A\lf(x; \frac{r}{2}, \frac{3r}{4}\r); B(x, r)\r)\geq cr^{n-q}.
\end{equation}

By Lemma \ref{lem:PhiHau}, for $\mathcal H^{n-k-1}$-almost every $x\in\{0\}\times\rr^{n-k-1}$, there exists $M_x<\infty$ with
 \begin{equation}\label{eq:EQU3}
\Phi(B(x, r))\leq M_xr^{n-k-1}.
\end{equation}

If $k\geq 2$, by inserting (\ref{equa:pq}), (\ref{eq:EQU2}) and (\ref{eq:EQU3}) into the inequality (\ref{eq:capden}),  we obtain the optimal bound in \eqref{eq:EQU1}, modulo logarithmic terms.
The case $k=1$ is analogous.

In conclusion, there is no hope in improving on the boundedness of the set function $\Phi$ from \eqref{eq:capden} so as to
obtain estimates that would hold at every boundary point. Moreover, the additivity of $\Phi$ gives
rather optimal measure density properties for points outside exceptional sets.


\end{proof}}

{
\section{Proof of Proposition \ref{le:stro}}
\begin{proof}[Proof of Proposition \ref{le:stro}]

Assuming that $\boz$ is $p$-capacitory dense at the point $z$ for $1\leq p<\fz$, there exists a positive constant $\delta_z>0$ and a decreasing positive sequence $\{r_i\}_{i=1}^{\fz}$, which converges to $0$, such that 
\begin{equation}\label{eq:capin0}
\frac{Cap_p\lf(\boz\cap B\lf(z,\frac{r_i}{4}\r), \boz\cap A\lf(z; \frac{r_i}{2}, \frac{3r_i}{4}\r); B\lf(z, r_i\r)\r)}{Cap_p\lf(B\lf(z, \frac{r_i}{4}\r), A\lf(z; \frac{r_i}{2}, \frac{3r_i}{4}\r); B(z, r_i)\r)}>\delta_z
\end{equation}
for every $r_i$. 

Let us first consider the case $p=1$. Since 
\begin{multline}
\mathcal W_1\lf(\boz\cap B\lf(z, \frac{r_i}{4}\r),A\lf(z; \frac{r_i}{2}, \frac{3r_i}{4}\r); B(z, r_i)\r)\nonumber\\
\subset\mathcal W_1\lf(\boz\cap B\lf(z, \frac{r_i}{4}\r),\boz\cap A\lf(z; \frac{r_i}{2}, \frac{3r_i}{4}\r); B(z, r_i)\r),
\end{multline}
we have 
\begin{multline}
Cap_1\lf(\boz\cap B\lf(z, \frac{r_i}{4}\r), A\lf(z; \frac{r_i}{2}, \frac{3r_i}{4}\r); B(z, r_i)\r)\nonumber\\
\geq Cap_1\lf(\boz\cap B\lf(z, \frac{r_i}{4}\r), \boz\cap A\lf(z;\frac{r_i}{2}, \frac{3r_i}{4}\r); B(z, r_i)\r).
\end{multline}
By \cite[Proposition 6.4]{GResh} we have that 
\[Cap_1\lf(B\lf(z, \frac{r_i}{4}\r), A\lf(z; \frac{r_i}{2}, \frac{3r_i}{4}\r); B(z, r_i)\r)\sim_c r_i^{n-1}\]
with an implicit constant independent of $r_i$. Hence we have 
\begin{equation}
\frac{r_iCap_1\lf(\boz\cap B\lf(z,\frac{r_i}{4}\r), A\lf(z; \frac{r_i}{2}, \frac{3r_i}{4}\r); B\lf(z, r_i\r)\r)}{\mathcal H^n(B(z, r_i))}>\tilde\delta_z>0.\nonumber
\end{equation}
This implies that $\boz$ is $1$-fat at $z$.

Let now $1<p<\fz$. Without loss of generality, we may choose a sequence $\{r_i\}_{i=1}^\fz$ with $16r_{i+1}<r_{i}$ for every $i\in\mathbb N$ such that (\ref{eq:capin0}) holds. By (\ref{equa:cap}), we have 
\begin{equation}\label{eq:capin}
Cap_{p}\lf(B\lf(z, \rho\r), A\lf(z; 2\rho, 3\rho\r); B(z, 4\rho)\r)\sim_cCap_p\lf(B\lf(z, \frac{r_i}{4}\r), A\lf(z; \frac{r_i}{2}, \frac{3r_i}{4}\r); B(z, r_i)\r)\end{equation}
for every $\rho\in(\frac{r_i}{4}, \frac{r_i}{2})$ with a constant $c$ independent of $\rho$ and $r_i$. Since $\rho\in(\frac{r_i}{4},\frac{r_i}{2})$, 
\[\mathcal W_p\lf(\boz\cap B\lf(z, \rho\r), A\lf(z; 2\rho, 3\rho\r); B(z, 4\rho)\r)\subset\mathcal W_p\lf(\boz\cap B\lf(z,\frac{r_i}{4}\r), A\lf(z; 2\rho, 3\rho\r); B(z, 4\rho)\r).\]
Hence, we have
\begin{multline}\label{eq:capin1}
Cap_p\lf(\boz\cap B\lf(z, \frac{r_i}{4}\r), A\lf(z; 2\rho, 3\rho\r); B(z, 4\rho)\r)\\ \leq Cap_p\lf(\boz\cap B(z,\rho), A\lf(z; 2\rho, 3\rho\r); B(z, 4\rho)\r).
\end{multline}
Let $u\in \mathcal W_p\lf(\boz\cap B\lf(z, \frac{r_i}{4}\r), A\lf(z; 2\rho, 3\rho\r); B(z, 4\rho)\r)$ be arbitrary. Then we define a function $\tilde u\in\mathcal W_p\lf(\boz\cap B\lf(z, \frac{r_i}{4}\r), \boz\cap A\lf(z; \frac{r_i}{2}, \frac{3r_i}{4}\r); B(z, r_i)\r)$ by setting 
$$
\tilde u(x):=\begin{cases}
u(x) &\textnormal{if }\ x\in B(z, \frac{r_i}{4})\\
u\lf((x-z)\frac{8\rho-r_i}{r_i}+\lf(\frac{r_i}{2}-2\rho\r)\frac{x-z}{|x-z|}+z\r) &\textnormal{if }\ x\in B\lf(z, \frac{r_i}{2}\r)\setminus B\lf(z, \frac{r_i}{4}\r)\\
u\lf(\frac{4\rho}{r_i}(x-z)+z\r) &\textnormal{if }\ x\in B(z, r_i)\setminus B\lf(z, \frac{r_i}{2}\r)\, \end{cases}.
$$
By the fact that $\frac{r_i}{4}\leq \rho\leq\frac{r_i}{2}$, we have 
\begin{equation}
\int_{B(z, r_i)}|\nabla\tilde u(x)|^pdx\leq C\int_{B(z, 4\rho)}|\nabla u(x)|^pdx\nonumber
\end{equation}
with a constant $C$ independent of $z$, $\boz$ and $\rho\in(\frac{r_i}{4}, \frac{r_i}{2})$. Since the test function $u$ was arbitrary, we have 
\begin{multline}\label{eq:capin2}
Cap_{p}\lf(\boz\cap B\lf(z, \frac{r_i}{4}\r), \boz\cap A\lf(z; \frac{r_i}{2}, \frac{3r_i}{4}\r); B(z, r_i)\r)\\
\leq CCap_p\lf(\boz\cap B\lf(z, \frac{r_i}{4}\r), A\lf(z; 2\rho, 3\rho\r); B(z, 4\rho)\r)
\end{multline}
with an absolute positive constant $C$ independent of $\rho\in(\frac{r_i}{4}, \frac{r_i}{2})$.
By combining inequalities (\ref{eq:capin1}) and (\ref{eq:capin2}), we obtain 
\begin{multline}\label{eq:capin3}
Cap_{p}\lf(\boz\cap B\lf(z, \frac{r_i}{4}\r), \boz\cap A\lf(z; \frac{r_i}{2}, \frac{3r_i}{4}\r); B(z, r_i)\r)\\
\leq CCap_p\lf(\boz\cap B\lf(z, \rho\r), A\lf(z; 2\rho, 3\rho\r); B(z, 4\rho)\r)
\end{multline}
with a positive constant $C$ independent of $\rho\in(\frac{r_i}{4}, \frac{r_i}{2})$. Finally, by combining inequalities (\ref{eq:capin0}), (\ref{eq:capin}) and (\ref{eq:capin3}), we obtain 
\[\frac{Cap_p\lf(\boz\cap B(z, \rho), A\lf(z; 2\rho, 3\rho\r); B(z, 4\rho)\r)}{Cap_p\lf(B(z, \rho), A\lf(z; 2\rho, 3\rho\r); B(z, 4\rho)\r)}>\tilde\delta_z\]
where $\tilde\delta_z>0$ is a positive constant independent of $\rho\in (\frac{r_i}{4}, \frac{r_i}{2})$. Since $16r_{i+1}<r_i$ for every $i\in\mathbb N$, we have
\begin{multline}
\int_0^1\lf(\frac{Cap_p\lf(\boz\cap B(z, \rho), A\lf(z; 2\rho, 3\rho\r); B(z, 4\rho)\r)}{Cap_p\lf(B(z, \rho), A\lf(z; 2\rho, 3\rho\r); B(z, 4\rho)\r)}\r)^{\frac{1}{p-1}}\frac{d\rho}{\rho}\nonumber\\
\geq\sum_{i=1}^\fz\int_{\frac{r_i}{4}}^{\frac{r_i}{2}}\lf(\frac{Cap_p\lf(\boz\cap B(z, \rho), A\lf(z; 2\rho, 3\rho\r); B(z, 4\rho)\r)}{Cap_p\lf(B(z, \rho), A\lf(z; 2\rho, 3\rho\r); B(z, 4\rho)\r)}\r)^{\frac{1}{p-1}}\frac{d\rho}{\rho}\\
\geq\sum_{i=1}^\fz\frac{(\tilde\delta_z)^{\frac{1}{p-1}}}{2}=\fz.
\end{multline}
Hence, $\boz$ is $p$-fat at the point $z$.

Next, for $1<p\leq n-1$, we construct outward cusp domains $\boz^n_w\subset\rn$ with suitable functions $w$, such that $\boz^n_w$ are $p$-fat but not $p$-capacitory dense at the tip $0$. 

Fix $1<p<n-1.$  We consider the function $w(t)=\frac{t}{\log^{\frac{p-1}{n-1-p}}\frac{e}{t}}$ and the corresponding outward cusp domain $\boz^n_w$. By Theorem \ref{lem:cacusp}, we have 
\begin{multline}
Cap_p\lf(\boz^n_w\cap B\lf(0,\frac{r}{4}\r),  A\lf(0; \frac{r}{2}, \frac{3r}{4}\r); B(0, r)\r)\sim_c\nonumber\\
Cap_p\lf(\boz^n_w\cap B\lf(0, \frac{r}{4}\r), \boz^n_w\cap A\lf(0; \frac{r}{2}, \frac{3r}{4}\r); B(0, r)\r)\sim_c\frac{r^{n-p}}{\log^{p-1}\frac{e}{r}}.
\end{multline}
Hence, by (\ref{equa:cap}), we have 
\[\lim_{r\to0^+}\frac{Cap_p\lf(\boz^n_w\cap B\lf(0, \frac{r}{4}\r), \boz^n_w\cap A\lf(0; \frac{r}{2}, \frac{3r}{4}\r); B(0, r)\r)}{Cap_p\lf(B\lf(0, \frac{r}{4}\r), A\lf(0; \frac{r}{2}, \frac{3r}{4}\r); B(0, r)\r)}\sim_c\lim_{r\to0^+}\frac{1}{\log^{p-1}\frac{e}{r}}=0\]
and 
\[\int_0^1\lf(\frac{Cap_p\lf(\boz^n_w\cap B\lf(0, \frac{r}{4}\r),  A\lf(0; \frac{r}{2}, \frac{3r}{4}\r); B(0, r)\r)}{Cap_p\lf(B\lf(0, \frac{r}{4}\r),  A\lf(0; \frac{r}{2}, \frac{3r}{4}\r); B(0, r)\r)}\r)^{\frac{1}{p-1}}\sim_c\int_{0}^1\frac{1}{r\log\frac{e}{r}}dr=\fz.\]
Hence, the outward cusp domain $\boz^n_w$ is not $p$-capacitory dense but nevertheless $p$-fat at the tip $0$.

For $p=n-1$, we choose the function $w(t)=t^2$. By Theorem \ref{lem:cacusp}, for every $0<r<\frac{1}{2}$, we have 
\begin{multline}
Cap_{n-1}\lf(\boz^n_w\cap B\lf(0,\frac{r}{4}\r),  A\lf(0; \frac{r}{2}, \frac{3r}{4}\r); B(0, r)\r)\sim_c\nonumber\\
Cap_{n-1}\lf(\boz^n_w\cap B\lf(0, \frac{r}{4}\r), \boz^n_w\cap A\lf(0; \frac{r}{2}, \frac{3r}{4}\r); B(0, r)\r)\sim_c\frac{r}{\log^{n-2}\frac{e}{r}}.
\end{multline}
Hence, we have 
\[\lim_{r\to0^+}\frac{Cap_{n-1}\lf(\boz^n_w\cap B\lf(0, \frac{r}{4}\r), \boz^n_w\cap A\lf(0; \frac{r}{2}, \frac{3r}{4}\r); B(0, r)\r)}{Cap_{n-1}\lf(B\lf(0, \frac{r}{4}\r), A\lf(0; \frac{r}{2}, \frac{3r}{4}\r); B(0, r)\r)}\sim_c\lim_{r\to0^+}\frac{1}{\log^{n-2}\frac{e}{r}}=0\]
and 
\[\int_0^{\frac{1}{2}}\lf(\frac{Cap_{n-1}\lf(\boz^n_w\cap B\lf(0, \frac{r}{4}\r),  A\lf(0; \frac{r}{2}, \frac{3r}{4}\r); B(0, r)\r)}{Cap_{n-1}\lf(B\lf(0, \frac{r}{4}\r),  A\lf(0; \frac{r}{2}, \frac{3r}{4}\r); B(0, r)\r)}\r)^{\frac{1}{n-2}}\frac{dr}{r}\sim_c\int_{0}^{\frac{1}{2}}\frac{1}{r\log\frac{e}{r}}dr=\fz.\]
Consequently, the outward cusp domain $\boz^n_w$ is not $(n-1)$-capacitory dense but nevertheless $(n-1)$-fat at the tip $0$. 
\end{proof} 

}

\section{Proofs of Theorem \ref{thm:positive} and Theorem \ref{thm:caF}}

In this section, for every $n\geq 3$ and $1\leq q<n-1$, we construct a Sobolev $(p, q)$-extension domain $\boz\subset\rn$ with $|\partial\boz|>0$. We also use this construction to prove Theorem \ref{thm:caF}.

\subsection{The initial construction}

Let $\mathcal Q_o:=(0,1)^n$ be the $n$-dimensional unit cube in $\rr^n$, and $\mathcal C_o:=(0,1)^{n-1}\times(0,2)$ be an $n$-dimensional rectangle in $\rr^{n}$. Let $\mathcal S_o:=(0, 1)^{n-1}$ be the $(n-1)$-dimensional unit cube in the $(n-1)$-dimensional Euclidean hyperplane $\rr^{n-1}$. Let $E\subset [0, 1]$ be a Cantor set with $0<\mathcal H^1(E)<1$. The Smith-Volterra-Cantor set guarantees the existence of such an $E$, see \cite{Smith}. Define 
$$E^{n-1}:=\underbrace{E\times E\times\cdots\times E}_{n-1}.$$
Then $E^{n-1}\subset[0, 1]^{n-1}$ is nowhere dense in $(0, 1)^{n-1}$ with $0<\mathcal H^{n-1}(E^{n-1})<1$. We let
$$W:=\lf\{Q\subset(0, 1)^{n-1}\setminus E^{n-1}: Q\ {\rm is\ Whitney}\r\}$$
 be the class of all Whitney cubes of the open set $(0, 1)^{n-1}\setminus E^{n-1}$, see \cite{Stein1}. For every $k\in\mathbb N$, we define $W_k$ to be the subclass of $W$ with
$$W_k:=\lf\{Q\subset W: 2^{-k-1}\leq l(Q)<2^{-k}\r\}$$
where $l(Q)$ is the edge-length of the cube $Q$. We number the elements in $W_k$ by 
$$W_k=\{ Q_k^{j}:1\le j\le N_k \}.$$
Notice that $N_k\le 2^{(n-1)(k+1)}$. For a Whitney cube $Q_k^{j}$, we refer to its center by $x_k^{j}.$  Let $h:[0,1]\to[0,1]$ be an increasing and continuous function with $h(0)=0$ and $h(t)>0$ when $t>0.$ We define 
\begin{equation}\label{rkoot}
r_k:=\lf(2^{-(n-1)(k+1)-k}h(8^{-k})\r)^{\frac{1}{n-1-q}},
\end{equation}
$\mathsf D_k^{j}:=B^{n-1}(x_k^{j}, r_k)$ and $\widetilde{\mathsf D}_k^{j}:=B^{n-1}(x_k^{j}, \frac{r_k}{2})$. Then $\widetilde{\mathsf D}_k^{j}\subset\mathsf D_k^{j}\subset Q_k^{j}$. 
Since $E^{n-1}$ is nowhere dense in $[0, 1]^{n-1}$, for an arbitrary $x\in E^{n-1}$ and each $\epsilon>0$, there exists a large enough $k$ and some $j\in\{1,2,\cdots, N_k\}$ with $Q_k^{j}\subset(0, 1)^{n-1}\cap B^{n-1}(x, \epsilon)$. Then $\widetilde{\mathsf D}_k^{j}\subset(0, 1)^{n-1}\cap B^{n-1}(x, \epsilon)$. Hence, we have 
  $$E^{n-1}\subset \overline{\bigcup_{k=1}^{\fz}\bigcup_{j}\widetilde{\mathsf D}_k^{j}}.$$ 
We define 
$$\mathcal D_h:=\bigcup_{k=1}^{\fz}\bigcup_{j}\mathsf D_k^{j} \ {\rm and}\ \widetilde{\mathcal D}_h:=\bigcup_{k=1}^\fz\bigcup_{j}\widetilde{\mathsf D}_k^{j}.$$
Then $E^{n-1}\subset\partial{\widetilde{\mathcal{D}}_h}$ and $\mathcal H^{n-1}(\partial\widetilde{\mathcal D}_h)\geq\mathcal H^{n-1}(E^{n-1})>0$.
\begin{figure}[htbp]
\centering
\includegraphics[width=0.5\textwidth]
{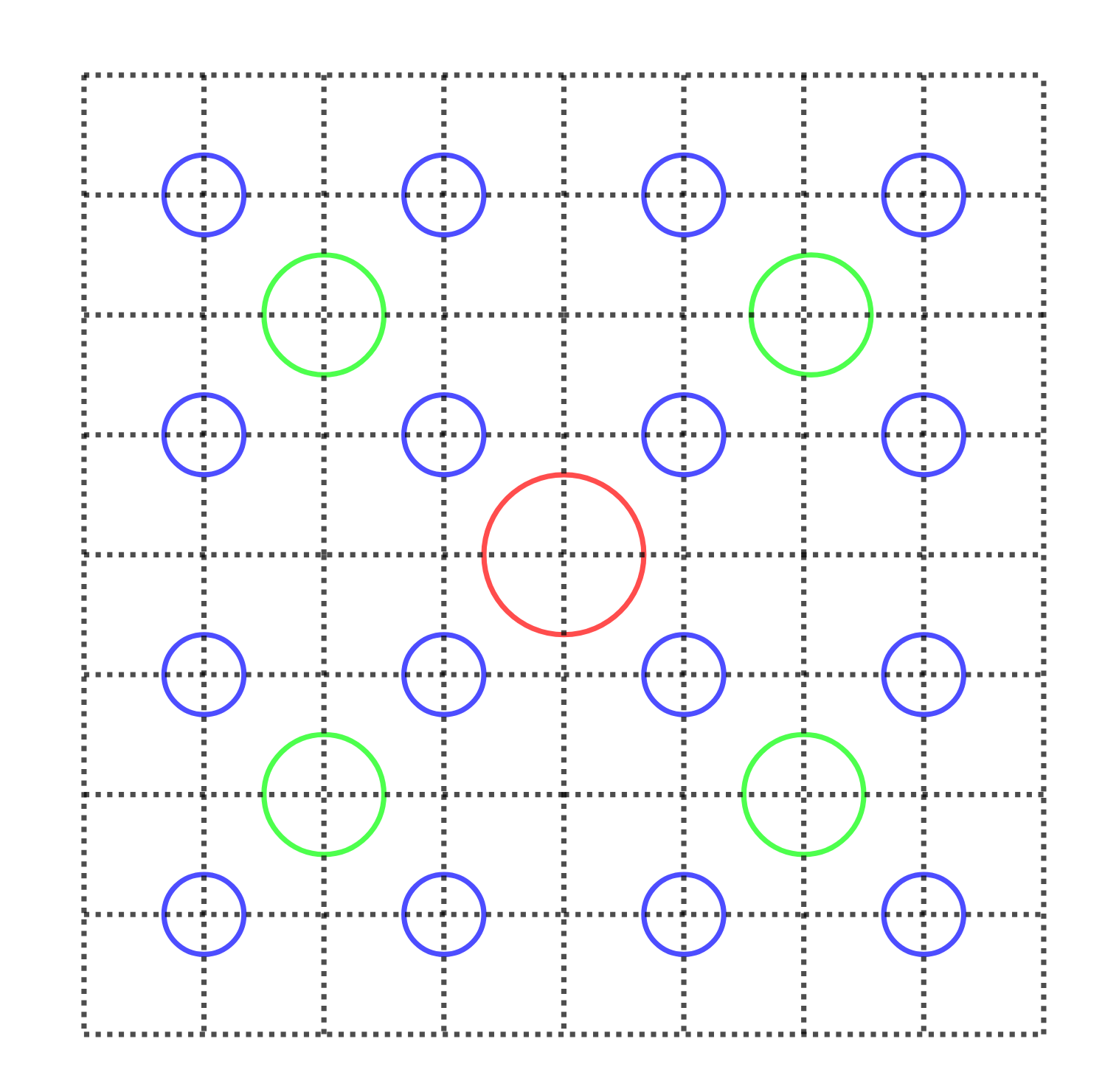}\label{densedisk}
\caption{The set $\mathcal D_h$}
\end{figure}

We define $\sC_k^{j}:=\mathsf D_k^{j}\times [1,2)$, $\widetilde \sC_k^{j}:=\widetilde{\mathsf D}_k^{j}\times[1,2)$ and $ A_k^{j}:=C_k^{j}\setminus\widetilde \sC_k^{j}$. We use the cylinders $\sC_k^{j}$ and $\widetilde \sC_k^{j}$ to define two domains:
$$\boz_h:=\mathcal Q_o\cup\bigcup_{k=1}^\fz\bigcup_{j}\sC_k^{j}\ {\rm and}\ \widetilde\boz_h:=\mathcal Q_o\cup\bigcup_{k=1}^\fz\bigcup_{n_k}\widetilde \sC_k^{j}.$$
 Given $m\in\mathbb N$, we set $$\boz_h^m:=\mathcal Q_o\cup\bigcup_{k=1}^m\bigcup_{j}\mathsf C_k^{j}\ {\rm and}\ \widetilde{\boz}_h^m:=\mathcal Q_o\cup\bigcup_{k=1}^m\bigcup_{j}\widetilde{\mathsf C}_k^{j}.$$  Figure $3$ illustrates the construction of these domains.

The following lemma goes back to a result of V\"ais\"al\"a \cite{vaisala}. See \cite[Pages 93-94]{vaisalafree} for a full proof. Also see \cite{Heinonen}.

\begin{lem}\label{lem:7qsb}
The domain $\widetilde{\boz}_h$ is quasiconformally equivalent to the unit ball: there is a quasiconformal mapping from the unit ball $B^n(0, 1)$ onto $\widetilde{\boz}_h$. 
\end{lem}
Hence, by Lemma \ref{lem:qbgh} and Lemma \ref{lem:density}, for arbitrary $1\leq p<\fz$, $W^{1,\fz}(\widetilde{\boz}_h)$ is dense in $W^{1, p}(\widetilde{\boz}_h)$.
Consequently, also $W^{1,\fz}(\widetilde{\boz}_h)\cap C(\widetilde{\boz}_h)$ is dense. 
\vskip 0.5 cm


\subsection{Cut-off functions}

Let $\mathsf C:=B^{n-1}(0, r)\times(0,1)$ and $\widetilde{\mathsf C}:=B^{n-1}(0, \frac{r}{2})\times(0, 1)$. Then $\mathsf C$ is a cylinder and $\widetilde{\mathsf C}$ is a sub-cylinder of $\mathsf C$. We define $A_\sC:=\sC\setminus\overline{\widetilde \sC}$. We employ the cylindrical coordinate system $$\{x=(x_1,x_2,\cdots,x_n)=(s, \theta_1, \theta_2, \cdots, \theta_{n-2}, x_n )\in\rn\}$$ where $\{(0, 0,\cdots, 0, x_n):x_n\in\rr\}$ is the rotation axis and $s=\sqrt{\sum_{i=1}^{n-1}x_i^2}$. For simplicity of notation, we write $\overrightarrow \theta=(\theta_1, \theta_2,\cdots, \theta_{n-2})$. Under this cylindrical coordinate system, we can write
$$\sC=\lf\{x=(s, \overrightarrow\theta, x_n)\in\rn;  x_n\in(0,1), s\in[0, r), \overrightarrow\theta\in[0, 2\pi)^{n-2}\r\},$$
$$\widetilde\sC=\lf\{x=(s, \overrightarrow\theta, x_n)\in\rn;  x_n\in(0,1), s\in\lf[0, \frac{r}{2}\r), \overrightarrow\theta\in[0, 2\pi)^{n-2}\r\},$$
and
$$A_\sC=\lf\{x=(s, \overrightarrow\theta, x_n,)\in\rn; x_n\in(0,1),  s\in\lf(\frac{r}{2}, r\r), \overrightarrow\theta\in[0, 2\pi)^{n-2}\r\}.$$
We define a subset $D_\sC$ of the cylinder $\sC$ by setting 
$$D_\sC:=\lf\{x=(s, \overrightarrow\theta, x_n)\in\rn; x_n\in\lf(0, \frac{r}{2}\r), s\in\lf(\frac{r}{2}, r-x_n\r), \overrightarrow\theta\in[0, 2\pi)^{n-2}\r\}.$$

\begin{figure}[htbp]
\centering
\includegraphics[width=0.4\textwidth]
{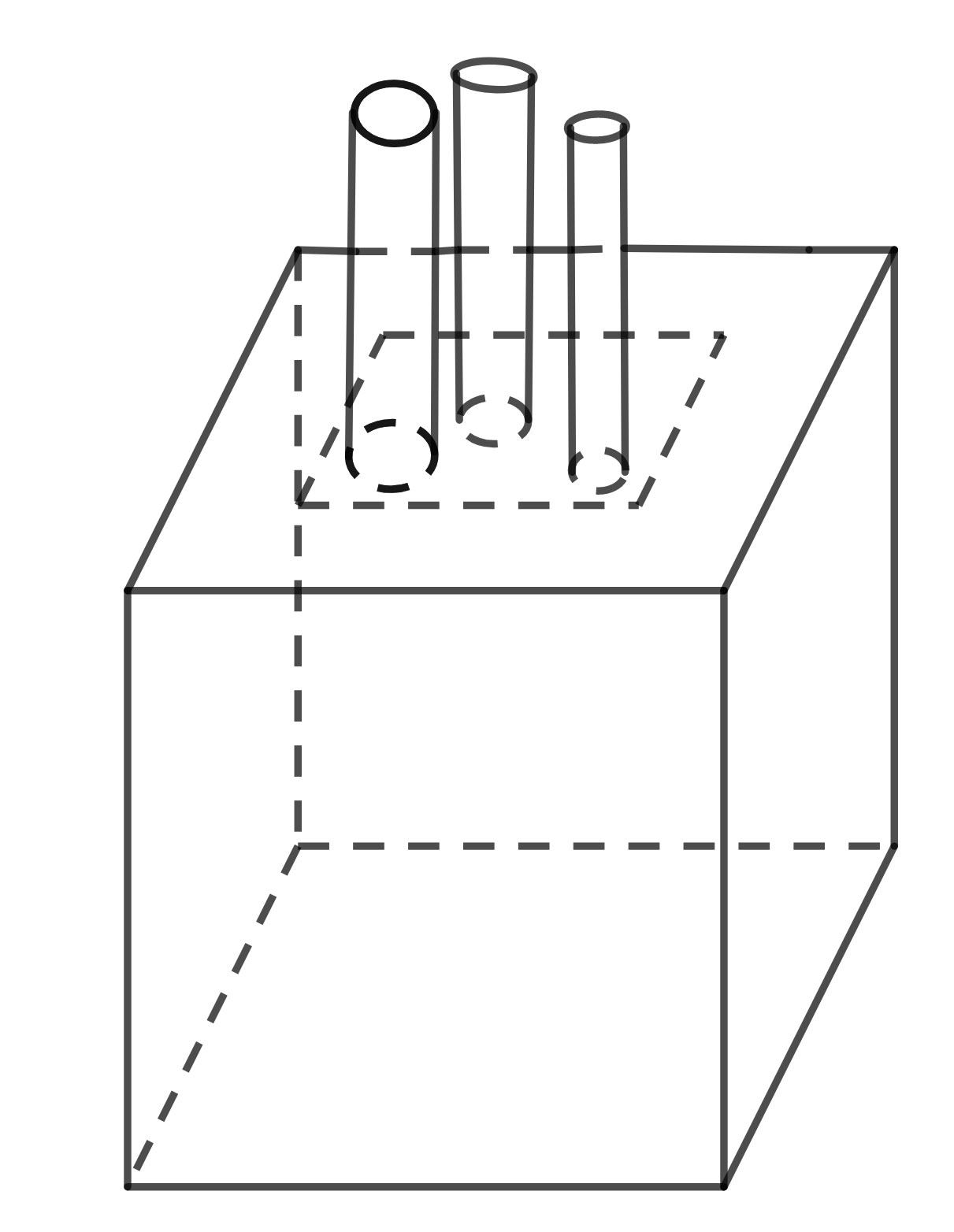}\label{fig:boz2}
\caption{The domain $\boz^\lambda_3$}
\end{figure} 

The next lemma gives two cut-off functions towards the construction of the desired extension operator.

\begin{lem}\label{lem:cut}
$(1):$ There exists a function $L^i_{\sC}: \overline{A_\sC}\to [0, 1]$ which is continuous on $\overline{A_\sC}\setminus\partial B^{n-1}(0, \frac{r}{2})\times\{0\}$, which equals zero both on $\lf(\overline B^{n-1}(0,r)\setminus B^{n-1}(0, \frac{r}{2})\r)\times \{0\}$ and on the set $\partial B^{n-1}(0, r)\times (0, 1)$, which equals $1$ on $\partial B^{n-1}(0, \frac{r}{2})\times(0,1)$ and which has the following additional properties. The function $L^i_\sC$ is Lipschitz on $A_\sC\setminus\overline{D_\sC}$ with 
$$|\nabla L^i_\sC(x)|\leq \frac{C}{r}\ {\rm for}\ x\in A_\sC\setminus\overline{D_\sC},$$
and $L^i_\sC$ is locally Lipschitz on $ D_\sC$ with
$$|\nabla L^i_\sC(x)|\leq\frac{C}{\sqrt{(s-\frac{r}{2})^2+x_n^2}}\ {\rm for}\ x\in D_\sC.$$
\\ 
$(2):$ There exists a function $L^o_\sC:\overline{ A_\sC}\to[0, 1]$, which is continuous on $\overline{A_\sC}\setminus\partial B^{n-1}(0,\frac{r}{2})\times\{0\}$, which equals zero on $\partial B^{n-1}(0, \frac{r}{2})\times[0,1)$, and which equals $1$ both on $\partial B^{n-1}(0,r)\times (0,1)$ and on
$\lf(B^{n-1}(0, r)\setminus\overline {B}^{n-1}(0, \frac{r}{2})\r)\times\{0\},$ 
and which has the additional following properties. The function $L^o_\sC$ is Lipschitz on $ A_\sC\setminus\overline{D_\sC}$ with
$$|\nabla L^o_\sC(x)|\leq\frac{C}{r}\ {\rm for}\ x\in  A_\sC\setminus\overline{D_\sC},$$
 and $L^o_\sC$ is locally Lipschitz on $D_\sC$ with
$$|\nabla L^o_\sC(x)|\leq\frac{C}{\sqrt{(s-\frac{r}{2})^2+x_n^2}}\ {\rm for}\ x\in D_\sC.$$
\end{lem}
\begin{proof}
$(1):$ We define the cut-off function $L^i_\sC$ on $\overline{A_\sC}$ with respect to the cylindrical coordinate system $\{x=(s, \overrightarrow\theta, x_n)\in\rn\}$ by setting
\begin{equation}\label{eq:cut2}
L^i_\sC(x)=\begin{cases} \frac{-2}{r}s+2, & x\in \overline{ A_\sC}\setminus\overline{D_\sC}, \\
 \frac{x_n}{x_n+(s-\frac{r}{2})}, &  x\in\overline{ D_\sC}\setminus\partial B^{n-1}(0, \frac{r}{2})\times\{0\},\\
 0, & x\in\partial B^{n-1}(0,\frac{r}{2})\times\{0\} \, . \end{cases}
\end{equation}
Then, if  $x\in  A_\sC\setminus\overline{D_\sC}$, we have 
$$
\frac{\partial L^i_\sC(x)}{\partial \theta_1}=...=\frac{\partial L^i_\sC(x)}{\partial \theta_{n-2}}=\frac{\partial L^i_\sC(x)}{\partial x_n}=0
\,\,\text{and}\,\,\left|\frac{\partial L^i_\sC(x)}{\partial s}\right|=\frac{2}{r}.
$$
If $x\in{D_\sC},$ we have
$$
\frac{\partial L^i_\sC(x)}{\partial \theta_1}=...=\frac{\partial L^i_\sC(x)}{\partial \theta_{n-2}}=0,
$$
$$
\left|\frac{\partial L^i_\sC(x)}{\partial s}\right|=\left|\frac{x_n}{\left(x_n+(s-\frac{r}{2})\right)^2}\right|\leq
\left|\frac{x_n+(s-\frac{r}{2})}{\left(x_n+(s-\frac{r}{2})\right)^2}\right|\leq \frac{1}{\sqrt{(s-\frac{r}{2})^2+x_n^2}}
$$
and
$$
\left|\frac{\partial L^i_\sC(x)}{\partial x_n}\right|=\left|\frac{(s-\frac{r}{2})}{\left(x_n+(s-\frac{r}{2})\right)^2}\right|\leq
\left|\frac{x_n+(s-\frac{r}{2})}{\left(x_n+(s-\frac{r}{2})\right)^2}\right|\leq \frac{1}{\sqrt{(s-\frac{r}{2})^2+x_n^2}}.
$$
Hence, we obtain 
\begin{equation}\label{equa:gra1}
|\nabla L^i_\sC(x)|\leq\begin{cases} \frac{C}{r}, & x\in A_\sC\setminus\overline{D_\sC},\\
\frac{C}{\sqrt{(s-\frac{r}{2})^2+x_n^2}}, & x\in {D_\sC} .
\end{cases}
\end{equation}

$(2):$ We define the cut-off function $L^o_\sC$ on $\overline{ A_\sC}$ with respect to the cylindrical coordinate system $\{x=(s, \overrightarrow\theta, x_n)\in\rn\}$ by setting 
\begin{equation}\label{eq:cut4}
L^o_\sC(x)=\begin{cases} \frac{2}{r}s-1, & x\in \overline{A_\sC}\setminus\overline{D_\sC}, \\
 \frac{s-\frac{r}{2}}{x_n+(s-\frac{r}{2})}, &  x\in\overline{D_\sC}\setminus\partial B^{n-1}(0, \frac{r}{2})\times\{0\},\\
 0, & x\in\partial B^{n-1}(0, \frac{r}{2})\times\{0\} \, . \end{cases}
\end{equation} 
By similar computations, we have
\begin{equation}\label{equa:gra3}
|\nabla L^o_\sC(x)|\leq \begin{cases} \frac{C}{r}, & x\in A_\sC\setminus\overline{D_\sC},\\
\frac{C}{\sqrt{x_n^2+(s-\frac{r}{2})^2}}, & x\in{D_\sC}.\end{cases}
\end{equation}
\end{proof}

\subsection{The extension operator}

Towards the construction of our extension operator, we define  piston-shaped domains $P_k^{j}$ by setting 
$$P_k^{j}:=\mathsf D_k^{j}\times(0, 1)\cup\widetilde{\mathsf D}_k^{j}\times[1,2).$$
The collection $\{P_k^{j}\}$ is pairwise disjoint. We set $U_1:=\mathcal S_o\times(1,2)\setminus\boz_h.$ 

 Given a cylinder $\sC_k^j$, in order to simplify our notation, we write $L^i_{k,j}= L^i_{\sC_k^{j}}$, $L^o_{k,j}=L^o_{\sC_k^{j}}$, $A_k^{j}=A_{\sC_k^{j}}$ and $ D_k^{j}= D_{\sC_k^{j}}$. Then we define cut-off functions $L^i$ and $L^o$ by setting
\begin{equation}\label{eq:tocut}
L^i(x):=\sum_{k,j}L^i_{k,j}(x)\ {\rm for}\ x\in\bigcup_{k, j}\overline{A_k^{j}}, 
\end{equation}
and
\begin{equation}\label{eq:tocut2}
L^o(x):=\sum_{k,j}L^o_{k,j}(x)\ \ {\rm for}\ \ x\in\bigcup_{k,j}\overline{ A_k^{j}}.
\end{equation}

We define a reflection on $\mathcal S_o\times(1, 2)$ by setting 
\begin{equation}\label{eq:ref1}
\mathcal R_1(x):=\lf(x_1,x_2,\cdots, x_{n-1}, 2-x_n\r)\ {\rm for\ every}\ x=(x_1,x_2,\cdots, x_n)\in\mathcal S_o\times(1, 2).
\end{equation}
On the set $\bigcup_{k,j} A_k^{j}$, we define a mapping $\mathcal R_2$ which is a reflection on every $A_k^{j}$. With respect to the local cylindrical coordinate system on every $A_k^{j}$, we write
\begin{equation}\label{eq:ref2}
\mathcal R_2(x):=\mathcal R_2(s, \overrightarrow\theta,  x_n)=\lf(-\frac s 2 +\frac 3 4 r_k, \overrightarrow\theta, x_n\r)
\end{equation}
for $x=(s, \overrightarrow\theta,  x_n)\in A_k^{j}$. Simple computations give the estimates
\begin{equation}\label{eq:jaco1}
\frac{1}{C}\leq|J_{\mathcal R_1}(x)|\leq C\ {\rm and}\ |D\mathcal R_1(x)|\leq C,
\end{equation}
for every $x\in\mathcal S_o\times(1, 2)$, and
\begin{equation}\label{eq:jaco2}
\frac{1}{C}\leq|J_{\mathcal R_2}(x)|\leq C\ {\rm and}\ |D\mathcal R_2(x)|\leq C,
\end{equation}
for every $x\in\bigcup_{k,j}A_k^{j}$. 

We begin by defining our linear extension operator on the dense subspace $W^{1,\fz}(\widetilde{\boz}_h)\cap C(\widetilde {\boz}_h)$ of $W^{1,p}(\widetilde{\boz}_h).$  Given $u\in W^{1,\fz}(\widetilde{\boz}_h)\cap C(\widetilde{\boz}_h)$, we define the extension $E(u)$ on the rectangle $\mathcal C_o$ by setting 
\begin{equation}\label{eq:exten}
E(u)(x):=\begin{cases} u(x), & x\in\widetilde{\boz}_h,\\
L^i(x)(u\circ\mathcal R_2)(x)+L^o(x)(u\circ\mathcal R_1)(x), & x\in \bigcup_{k,j}\overline{A_k^{j}},\\
(u\circ\mathcal R_1)(x), & x\in U_1\, . \end{cases}
\end{equation}
We continue with the local properties of our extension operator.
 
\begin{lem}\label{le:lopro}
Let $E$ be the extension operator defined in (\ref{eq:exten}). Then, for every $u\in W^{1, \fz}(\widetilde\boz_h)\cap C(\widetilde {\boz}_h)$, we have:

\textbf{$(1)$:} $E(u)$ is Lipschitz on $U_1$ with 
\begin{equation}\label{eq:din1}
|\nabla E(u)(x)|\leq |\nabla (u\circ \mathcal R_1)(x)|
\end{equation}
for almost every $x\in U_1$.

\textbf{$(2)$:} $E(u)$ is locally Lipschitz on ${A_k^{n_k}}$ with 
\begin{multline}\label{eq:din2}
|\nabla E(u)(x)|\leq |\nabla L^i_{k,n_k}(x)(u\circ\mathcal R_2)(x)|+|L^i_{k,n_k}(x)\nabla(u\circ\mathcal R_2)(x)|\\
+|\nabla L^o_{k,n_k}(x)(u\circ\mathcal R_1)(x)|+|L^o_{k,n_k}(x)\nabla(u\circ\mathcal R_1)(x)|
\end{multline}
for almost every $x\in{A_k^{j}}$. 

Moreover, with respect to the local cylindrical system $x=(s,\overrightarrow\theta, x_n)$ on $\mathsf C_k^j$, for every $1\leq q<\fz$, we have
\begin{equation}\label{eq:nin1}
\int_{\mathsf C_k^{j}}|E(u)(x)|^qdx\leq C\int_{P_k^{j}}|u(x)|^qdx
\end{equation}
 and 
 \begin{multline}\label{eq:nin2}
 \int_{\mathsf C_k^{j}}|\nabla E(u)(x)|^qdx\leq C\int_{P_k^{j}}|\nabla u(x)|^qdx\\
 +C\int_{D_k^{j}}\lf(\sqrt{\frac{1}{x_n^2+\lf(s-\frac{r_k}{2}\r)^2}}\r)^{q}(|u\circ\mathcal R_1(x)|^q+|u\circ\mathcal R_2(x)|^q)dx\\
 +C\int_{A_k^{j}\setminus\overline{D_k^{j}}}\lf(\frac{1}{r_k}\r)^{q}(|u\circ\mathcal R_1(x)|^q+|u\circ\mathcal R_2(x)|^q)dx,
 \end{multline}
 with some uniform positive constant $C$.
\end{lem}
\begin{proof}
Since $u\in W^{1,\fz}(\widetilde\boz_h)\cap C(\widetilde {\boz}_h)$, definitions of cut-off functions $L^i, L^o$ and reflections $\mathcal R_1, \mathcal R_2$ easily yield that $E(u)$ is Lipschitz on $U_1$ and that $E(u)$ is locally Lipschitz on $A_k^{j}$ for every $k$ and $j$. Inequalities (\ref{eq:din1}) and (\ref{eq:din2}) follow by the chain rule.

By the definition of $E(u)$ in (\ref{eq:exten}), we have
\begin{multline}\label{eq:Enin1}
\int_{\mathsf C_k^{j}}|E(u)(x)|^qdx\leq \int_{P_k^{j}}|u(x)|^qdx\\
+\int_{A_k^{j}}|L^i_{k,j}(x)(u\circ\mathcal R_2)(x)+L^o_{k,j}(x)(u\circ\mathcal R_1)(x)|^qdx.  
\end{multline}
Since $0\leq L^i_{k,j}(x)\leq 1$ and $0\leq L^o_{k,j}(x)\leq1$ for every $x\in A_k^{j}$, by (\ref{eq:jaco1}), (\ref{eq:jaco2}) and the change of variables formula, we have
\begin{multline}\label{eq:Enin2}
\int_{A_k^{j}}|L^i_{k,j}(x)(u\circ\mathcal R_2)(x)+L^o_{k,j}(x)(u\circ\mathcal R_1)(x)|^qdx\\\leq C\int_{A_k^{j}}|u\circ\mathcal R_1(x)|^qdx+C\int_{A_k^{j}}|u\circ\mathcal R_2(x)|^qdx\\
\leq C\int_{P_k^{j}}|u(x)|^qdx.
\end{multline}
By combining inequalities (\ref{eq:Enin1}) and (\ref{eq:Enin2}), we obtain inequality (\ref{eq:nin1}).

By inequality (\ref{eq:din2}), we have 
\begin{equation}\label{eq:nEnin1}
\int_{\mathsf C_k^{j}}|\nabla E(u)(x)|^qdx\leq \int_{P_k^{j}}|\nabla u(x)|^qdx+I^{k, j}_1+I^{k,j}_2,
\end{equation}
where
\begin{equation}\label{eq:nEnin2}
I^{k,j}_1:=\int_{A_k^{j}}|L^i_{k,j}(x)\nabla(u\circ\mathcal R_2)(x)|^qdx+
\int_{A_k^{j}}|L^o_{k,j}(x)\nabla(u\circ\mathcal R_1)(x)|^qdx\nonumber
\end{equation}
and
\begin{equation}\label{eq:nEnin3}
I^{k,j}_2:=\int_{A_k^{j}}|\nabla L^i_{k,j}(x)(u\circ\mathcal R_2)(x)|^qdx+\int_{A_k^{j}}|\nabla L^o_{k,j}(x)(u\circ\mathcal R_1)(x)|^qdx.\nonumber
\end{equation}
Arguing as for (\ref{eq:Enin2}), we have
\begin{equation}\label{eq:nEnin4}
I^{k,j}_1\leq C\int_{P_k^{j}}|\nabla u(x)|^qdx.
\end{equation}
By inequality (\ref{equa:gra1}), we have 
\begin{multline}\label{eq:nEnin5}
\int_{A_k^{j}}|\nabla L^i_{k,j}(x)(u\circ\mathcal R_2)(x)|^qdx\\
 \leq C\int_{A_k^{j}\setminus\overline{D_k^{j}}}\lf(\frac{1}{r_k}\r)^{q}|(u\circ\mathcal R_2)(x)|^qdx\\
 +C\int_{D_k^{j}}\lf(\sqrt{\frac{1}{x_n^2+\lf(s-\frac{r_k}{2}\r)^2}}\r)^{q}|(u\circ\mathcal R_2)(x)|^qdx.
  \end{multline}
 By (\ref{equa:gra3}), we have 
 \begin{multline}\label{eq:nEnin6}
\int_{A_k^{j}}|\nabla L^o_{k,j}(x)(u\circ\mathcal R_1)(x)|^qdx\\
 \leq C\int_{A_k^{j}\setminus\overline{D_k^{j}}}\lf(\frac{1}{r_k}\r)^{q}|(u\circ\mathcal R_1)(x)|^qdx\\
 +C\int_{D_k^{j}}\lf(\sqrt{\frac{1}{x_n^2+\lf(s-\frac{r_k}{2}\r)^2}}\r)^{q}|(u\circ\mathcal R_1)(x)|^qdx.
  \end{multline}
  In conclusion, (\ref{eq:nEnin5}) and (\ref{eq:nEnin6}) give
  \begin{multline}\label{eq:nEnin7}
  I^{k,j}_2\leq C\int_{D_k^{j}}\lf(\sqrt{\frac{1}{x_n^2+\lf(s-\frac{r_k}{2}\r)^2}}\r)^{q}((|u\circ\mathcal R_1)(x)|^q+|(u\circ\mathcal R_2)(x)|^q)dx\\
 +C\int_{A_k^{j}\setminus\overline{D_k^{j}}}\lf(\frac{1}{r_k}\r)^{q}(|(u\circ\mathcal R_1)(x)|^q+|(u\circ\mathcal R_2)(x)|^q)dx.
  \end{multline}
  Finally, by combining inequalities (\ref{eq:nEnin1}), (\ref{eq:nEnin4}) and (\ref{eq:nEnin7}), we obtain inequality (\ref{eq:nin2}).
\end{proof}

\subsection{An extension theorem}

The following theorem provides us with examples of irregular extension domains.

\begin{thm}\label{thm:BET}
Let $1\leq q<n-1$ and ${(n-1)q}/{(n-1-q)}<p<\fz$ be fixed. Given $\lambda>0,$ define
\begin{equation}\label{eq:defnh}
h_\lambda(t):=\lf(\frac{1}{t}\r)^{((1-\lambda(n-1-q))(n-1)(k+1)+k)/3k}.
\end{equation}
There exists  $\lambda_o:=\lambda_o(p, q)>0$ such that $\widetilde\boz_h\subset\rn$ is a 
Sobolev $(p, q)$-extension domain with $|\partial \widetilde \boz_h|>0$ whenever $h(t)\le h_\lambda(t)$ for some $\lambda>\lambda_o$ and all $0<t\le 1.$ 
\end{thm}
\begin{proof}
By the definition of $h_\lambda$ and \eqref{rkoot}, we have 
\begin{equation}\label{eq:radius}
r_k\leq 2^{-\lambda(n-1)(k+1)}.
\end{equation}
Set 
\begin{equation}\label{lambdaoo}
\lambda_o(p, q):=\max\lf\{\frac{n-1-p}{(n-1)^2}, \frac{p-q}{(n-1)(p-q)-pq}\r\}.
\end{equation}
Then, for every $\lambda>\lambda_o$, we have $1\leq q<\frac{((n-1)\lambda-1)p}{\lambda p+(n-1)\lambda-1}<n-1$. Fix such a $\lambda.$ To simplify our notation, we refer to $h_\lambda$ by $h$ in what follows.
Since $E^{n-1}\subset\partial\widetilde{\mathcal D}_h$ and $\mathcal H^{n-1}(E^{n-1})>0$, we have $E^{n-1}\times[1,2]\subset\partial\widetilde\boz_h$ and $\mathcal H^n(\partial\widetilde\boz_h)\ge \mathcal H^n(E^{n-1}\times[1,2])>0$. 

In order to prove that $E$ defined in (\ref{eq:exten}) is a bounded extension operator, we need an approximation argument. Given $u\in W^{1,\fz}(\widetilde\boz_h)\cap C(\widetilde \boz_h)$ and  $m\in\mathbb N$, we 
define $u_m:=u\big|_{\widetilde{\boz}_h^m}.$ Since $\widetilde{\boz}_h^m$ is clearly quasiconvex, it follows that $u_m$ is Lipschitz and bounded. We define the extension $E^m(u_m)$ of $u_m$ by setting 
\begin{equation}\label{eq:extenm}
E^m(u_m)(x):=\begin{cases} u_m(x), & x\in\widetilde{\boz}^m_h,\\
L^i(x)(u_m\circ\mathcal R_2)(x)+L^o(x) (u_m\circ\mathcal R_1)(x), & x\in \bigcup_{k=1}^m\bigcup_{j}\overline{A_k^{j}},\\
(u_m\circ\mathcal R_1)(x), & x\in U_1^m \, ,\end{cases}\nonumber
\end{equation}
where $U_1^m=\mathcal{S}_0\times (0,1)\setminus \boz_h^m.$
Since $u_m$ is Lipschitz, $E^m(u_m)$ is $ACL$ on $\mathcal C_o$. By the definition of $u_m$ and the H\"older inequality, we have 
\begin{equation}
\label{eq:in1}
\int_{\widetilde\boz_h^m}|u_m(x)|^qdx\leq\int_{\widetilde\boz_h}|u(x)|^qdx\leq C\lf(\int_{\widetilde\boz_h}|u(x)|^pdx\r)^{\frac{q}{p}}
\end{equation}
and
\begin{equation}\label{eq:in11}
\int_{\widetilde\boz_h^m}|\nabla u_m(x)|^qdx\leq\int_{\widetilde\boz_h}|\nabla u(x)|^qdx\leq C\lf(\int_{\widetilde\boz_h}|\nabla u(x)|^pdx\r)^{\frac{q}{p}}.
\end{equation}

Since the collection $\{P_k^{j}\}$ is pairwise disjoint, by summing over $j$ and $k$, (\ref{eq:nin1}) and the H\"older inequality imply
\begin{multline}\label{eq:norm1}
\int_{\bigcup_{k=1}^m\bigcup_{j}\mathsf C_k^{j}}|E^m(u_m)(x)|^qdx\leq C\int_{\bigcup_{k=1}^m\bigcup_{j}P_k^{j}}|u_m(x)|^qdx\leq C\lf(\int_{\widetilde \boz_h}|u(x)|^pdx\r)^{\frac{q}{p}}.
\end{multline}
By (\ref{eq:jaco1}), the change of variables formula and the H\"older inequality, we have 
\begin{equation}
\label{eq:norm3}
\int_{U_1^m}|u_m\circ\mathcal R_1(x)|^qdx
      \leq\int_{\mathcal R_1\lf(U_1^m\r)}|u_m(x)|^qdx
      \leq C\lf(\int_{\widetilde{\boz}_h}|u(x)|^pdx\r)^{\frac{q}{p}}.
\end{equation}
Consequently, by combining (\ref{eq:in1}), (\ref{eq:norm1}) and (\ref{eq:norm3}), we obtain 
\begin{equation}\label{eq:IN1}
\lf(\int_{\mathcal C_o}|E^m(u_m)(x)|^qdx\r)^{\frac{1}{q}}\leq C\lf(\int_{\widetilde{\boz}_h}|u(x)|^pdx\r)^{\frac{1}{p}},
\end{equation}
where the constant $C$ is independent of $m$ and $u$.

By (\ref{eq:din1}), we have 
\begin{equation}
\label{eq:E1}
\int_{U_1^m}|\nabla E^m(u_m)(x)|^qdx
\leq 
 \int_{U_1^m}|\nabla (u_m\circ\mathcal R_1)(x)|^qdx.
\end{equation}
By (\ref{eq:jaco1}), the change of variables formula and the H\"older inequality, we obtain 
\begin{multline}
\label{eq:E3}
\int_{U_1^m}|\nabla (u_m\circ\mathcal R_1)(x)|^qdx\leq\int_{U_1^m}|\nabla (u_m\circ\mathcal R_1)(x)|^qdx\\
                \leq C\int_{\mathcal R_1\lf(U_1^m\r)}|\nabla u_m(x)|^qdx
                \leq C\lf(\int_{\widetilde{\boz}_h}|\nabla u(x)|^pdx\r)^{\frac{q}{p}}.
\end{multline}
By combining (\ref{eq:E1}) and (\ref{eq:E3}), we obtain 
\begin{eqnarray}\label{eq:norm4}
\int_{U_1^m}|\nabla E^m(u_m)(x)|^qdx\leq C\lf(\int_{\widetilde{\boz}_h}|\nabla u(x)|^p dx\r)^{\frac{q}{p}},
\end{eqnarray}
where the constant $C$ is independent of $m$ and $u$. 

By (\ref{eq:nin2}) and the fact that the collection $\{P_k^{j}\}$ is pairwise disjoint, we have 
\begin{multline}\label{eq:Eein1}
\int_{\bigcup_{k=1}^m\bigcup_{j}\mathsf C_k^{j}}|\nabla E^m(u_m)(x)|^qdx\leq C\int_{\bigcup_{k=1}^m\bigcup_{j}P_k^{j}}|\nabla u_m(x)|^qdx\\
 +C\int_{\bigcup_{k=1}^m\bigcup_{j}D_k^{j}}\lf(\sqrt{\frac{1}{x_n^2+\lf(s-\frac{r_k}{2}\r)^2}}\r)^{q}(|(u_m\circ\mathcal R_1)(x)|^q+|(u_m\circ\mathcal R_2)(x)|^q)dx\\
 +C\int_{\bigcup_{k=1}^m\bigcup_{j}A_k^{j}\setminus\overline{D_k^{j}}}\lf(\frac{1}{r_k}\r)^{q}(|(u_m\circ\mathcal R_1)(x)|^q+|(u_m\circ\mathcal R_2)(x)|^q)dx.
\end{multline}

The H\"older inequality gives
\begin{equation}\label{eq:Eein2}
\int_{\bigcup_{k=1}^m\bigcup_{j}P_k^{j}}|\nabla u_m(x)|^qdx\leq C\lf(\int_{\widetilde\boz_h}|\nabla u(x)|^pdx\r)^{\frac{q}{p}},
\end{equation}
\begin{multline}\label{eq:Eein3}
\int_{\bigcup_{k=1}^m\bigcup_{j}D_k^{j}}\lf(\sqrt{\frac{1}{x_n^2+\lf(s-\frac{r_k}{2}\r)^2}}\r)^q
(|(u_m\circ\mathcal R_1)(x)|^q+|(u_m\circ\mathcal R_2)(x)|^q)dx\\
\leq C\lf(\int_{\bigcup_{k=1}^m\bigcup_{j}D_k^{j}}|(u_m\circ\mathcal R_1)(x)|^p+|(u_m\circ\mathcal R_2)(x)|^pdx\r)^{\frac{q}{p}}\\
\times\lf(\int_{\bigcup_{k=1}^m\bigcup_{j}D_k^{j}}\lf(\sqrt{\frac{1}{x_n^2+\lf(s-\frac{r_k}{2}\r)^2}}\r)^{\frac{pq}{p-q}}dx\r)^{\frac{p-q}{p}}
\end{multline}
and
\begin{multline}\label{eq:Eein4}
\int_{\bigcup_{k=1}^m\bigcup_{j}A_k^{j}\setminus\overline{D_k^{j}}}
\lf(\frac{1}{r_k}\r)^q(|(u_m\circ\mathcal R_1)(x)|^q+|(u_m\circ\mathcal R_2)(x)|^q)dx\\
\leq\lf(\int_{\bigcup_{k=1}^m\bigcup_{j}A_k^{j}\setminus\overline{D_k^{j}}}|(u_m\circ\mathcal R_1)(x)|^p+|(u_m\circ\mathcal R_2)(x)|^pdx\r)^{\frac{q}{p}}\\
\times\lf(\int_{\bigcup_{k=1}^m\bigcup_{j}A_{k}^{j}\setminus\overline{D_k^{j}}}
\lf(\frac{1}{r_k}\r)^{\frac{pq}{p-q}}dx\r)^{\frac{p-q}{p}}.
\end{multline}

By (\ref{eq:jaco1}) and (\ref{eq:jaco2}), the change of variables formula yields that
\begin{equation}\label{eq:Eein5}
\int_{\bigcup_{k=1}^m\bigcup_{j}A_k^{j}\setminus\overline{D_k^{j}}}|(u_m\circ\mathcal R_1)(x)|^p+|(u_m\circ\mathcal R_2)(x)|^pdx\leq C\int_{\widetilde\boz_h}|u(x)|^pdx
\end{equation}
and
\begin{equation}\label{eq:Eein6}
\int_{\bigcup_{k=1}^m\bigcup_{j}D_k^{j}}|(u_m\circ\mathcal R_1)(x)|^p+|(u_m\circ\mathcal R_2)(x)|^pdx\leq C\int_{\widetilde\boz_h}|u(x)|^pdx.
\end{equation}
With $l_k=\sqrt{x_n^2+\lf(s-\frac{r_k}{2}\r)^2}$, by (\ref{eq:radius}) and \eqref{lambdaoo}, we have
\begin{multline}\label{eq:Eein7}
\int_{\bigcup_{k=1}^m\bigcup_{j}D_k^{j}}\lf(\frac{1}{\sqrt{x_n^2+\lf(s-\frac{r_k}{2}\r)^2}}\r)^{\frac{pq}{p-q}}dx\leq C\sum_{k=1}^m\sum_{j}r_k\int_0^{r_k}l_k^{n-2-\frac{pq}{p-q}}dl_k\\
\leq C\sum_{k=1}^m\sum_{j}r_k^{n-\frac{pq}{p-q}}\leq C\sum_{k=1}^\fz2^{(n-1)(k+1)\lf(1-\lambda\lf(n-\frac{pq}{p-q}\r)\r)}<\fz.
\end{multline}
Furthermore,
\begin{multline}\label{eq:Eein8}
\int_{\bigcup_{k=1}^m\bigcup_{j}A_k^{j}\setminus\overline{D_k^{j}}}\lf(\frac{1}{r_k}\r)^{\frac{pq}{p-q}}dx\leq C\sum_{k=1}^m\sum_{j=1}^{N_k}r_k^{n-1-\frac{pq}{p-q}}\\\leq C\sum_{k=1}^\fz2^{(n-1)(k+1)\lf(1-\lambda\lf(n-1-\frac{pq}{p-q}\r)\r)}<\fz.
\end{multline}
By combining inequalities (\ref{eq:Eein1})-(\ref{eq:Eein8}), we deduce that 
\begin{equation}\label{eq:Eein9}
\int_{\bigcup_{k=1}^m\bigcup_{j}\mathsf C_k^{j}}|\nabla E^m(u_m)(x)|^qdx\leq C\lf(\int_{\widetilde\boz_h}|u(x)|^p+|\nabla u(x)|^pdx\r)^{\frac{q}{p}}.
\end{equation}
Next, by combining (\ref{eq:in11}), (\ref{eq:norm4}) and (\ref{eq:Eein9}), we conclude that
\begin{equation}\label{eq:Eein10}
\int_{\mathcal C_o}|\nabla E^m(u_m)(x)|^qdx\leq C\lf(\int_{\widetilde\boz_h}|u(x)|^p+|\nabla u(x)|^pdx\r)^{\frac{q}{p}}.
\end{equation}
Hence, by combining (\ref{eq:IN1}) and (\ref{eq:Eein10}), we infer that
\begin{equation}\label{eq:in3}
\|E^m(u_m)\|_{W^{1,q}(\mathcal C_o)}\leq C\|u\|_{W^{1,p}(\widetilde{\boz}_h)},
\end{equation} 
uniformly in $m$.

By the definitions of $u_m$ and $E^m(u_m)$, for arbitrary $m, m'\in\mathbb N$ with $m<m'$, we have 
\begin{multline}\label{eq:in4}
\|E^m(u_m)-E^{m'}(u_{m'})\|^q_{W^{1,q}(\mathcal C_o)}\leq \int_{\bigcup_{k=m+1}^{m'}\bigcup_{j}\mathsf C_k^{j}}\left(|E^m(u_m)(x)|^q+|\nabla E^m(u_m)(x)|^q\right)dx\\
                                                                          +\int_{\bigcup_{k=m+1}^{m'}\bigcup_{j}\mathsf C_k^{j}}\left(|E^{m'}(u_{m'})(x)|^q+|\nabla E^{m'}(u_{m'})(x)|^q\right)dx.  
\end{multline}
By the definition of $E^m(u_m)$ and $E^{m'}(u_{m'})$, the H\"older inequality implies 
\begin{multline}\label{eq:es1}
\int_{\bigcup_{k=m+1}^{m'}\bigcup_{j}\mathsf C_k^{j}}\left(|E^m(u_m)(x)|^q+|\nabla E^m(u_m)(x)|^q\right)dx\\
\leq C\int_{\mathcal R_1\lf(\bigcup_{k=m+1}^{m'}\bigcup_{j}\mathsf C_k^{j}\r)}\left(|u_m(x)|^q+|\nabla u_m(x)|^q\right)dx\\
\leq C(p,q)\lf(\int_{\mathcal R_1\lf(\bigcup_{k=m+1}^{m'}\bigcup_{j}\mathsf C_k^{j}\r)}\left(|u(x)|^p+|\nabla u(x)|^p\right)dx\r)^{\frac{q}{p}},
\end{multline}
and
\begin{multline}\label{eq:es2}
\int_{\bigcup_{k=m+1}^{m'}\bigcup_{j}\mathsf C_k^{j}}\left(|E^{m'}(u_{m'})(x)|^q+|\nabla E^{m'}(u_{m'})(x)|^q\right)dx\leq \\
C\lf(\int_{\bigcup_{k=m+1}^{m'}\bigcup_{j}\widetilde{\mathsf C}_k^{j}}\left(|u_{m'}|^p+|\nabla u_{m'}|^p\right)dx+\int_{\mathcal R_1\lf(\bigcup_{k=m+1}^{m'}\bigcup_{j} A_k^{j}\r)}\left(|u_{m'}|^p+|\nabla u_{m'}|^p\right)dx\r)^{\frac{q}{p}}\\
 \leq C\lf(\int_{\bigcup_{k=m+1}^{m'}\bigcup_{j}\widetilde{\mathsf C}_k^{j}}\left(|u|^p+|\nabla u|^p\right)dx+\int_{\mathcal R_1\lf(\bigcup_{k=m+1}^{m'}\bigcup_{j}A_k^{j}\r)}\left(|u|^p+|\nabla u|^p\right)dx\r)^{\frac{q}{p}}.
\end{multline}
Since the volumes of $\mathcal R_1\lf(\bigcup_{k=m+1}^{m'}\bigcup_{j}A_k^{j}\r)$ and of $\bigcup_{k=m+1}^{m'}\bigcup_{j}\widetilde{\mathsf C}_k^{j}$ tend to zero as $m, m'$ approach infinity, both terms in (\ref{eq:es1}) and (\ref{eq:es2})  converge to zero. Consequently, $\{E^m(u_m)\}$ is a Cauchy sequence in the Sobolev space $W^{1,q}(\mathcal C_o)$ and hence converges to some function $v\in W^{1,q}(\mathcal C_o)$ with respect to the $W^{1,q}$-norm. Furthermore, there exists a subsequence of $\{E^m(u_m)\}$ which converges to $v$ almost everywhere in $\mathcal C_o$. On the other hand, by the definitions of $E^m(u_m)$ and $E(u)$, we have 
$$\lim_{m\to\fz}E^m(u_m)(x)=E(u)(x)$$
for almost every $x\in\mathcal C_o$. Hence $v(x)=E(u)(x)$ almost everywhere. This implies that $E(u)\in W^{1,q}(\mathcal C_o)$ with
\begin{equation}\label{eq:E42}
\|E(u)\|_{W^{1,q}(\mathcal C_o)}=\|v\|_{W^{1,q}(\mathcal C_o)}=\lim_{m\to\fz}\|E^m(u_m)\|_{W^{1,q}(\mathcal C_o)}\leq C\|u\|_{W^{1,p}(\widetilde{\boz}_h)}.
\end{equation}
We conclude that $E$ defined in (\ref{eq:exten}) is a linear extension operator from $W^{1,\fz}(\widetilde{\boz}_h)\cap C(\widetilde{\boz}_h)$ to $W^{1,q}(\mathcal C_o)$ with the norm inequality
$$
\|E(u)\|_{W^{1,q}(\mathcal C_o)}\leq C\|u\|_{W^{1,p}(\widetilde\boz_h)},
$$
where $C$ is independent of $u$. Since $W^{1,\fz}(\widetilde{\boz}_h)\cap C(\widetilde{\boz}_h)$ is dense in $W^{1,p}(\widetilde{\boz}_h)$, we can extend $E$ to entire $W^{1,p}(\widetilde{\boz}_h)$. It follows that $\widetilde \boz_h$ is a Sobolev $(p, q)$-extension domain, since $\mathcal C_o$ is  a 
$(q,q)$-extension domain.
\end{proof}

\subsection{Proofs of Theorem \ref{thm:positive} and \ref{thm:caF}}

\begin{proof}[Proof of Theorem \ref{thm:positive}]
The claim is an immediate consequence of Theorem \ref{thm:BET}.
\end{proof}

\begin{proof}[Proof of Theorem \ref{thm:caF}]
Let $n\geq 3$ and $1\leq q<n-1$. Fix $(n-1)q/(n-1-q)<p<\fz$ and a strictly increasing and continuous function $h: [0, 1]\to[0, 1]$.
Fix $\lambda>\lambda_o,$ where $\lambda_o$ is from Theorem \ref{thm:BET}. Define  
$$\tilde h(t):=\min\lf\{h(t), \lf(\frac{1}{t}\r)^{((1-\lambda(n-1-q))(n-1)(k+1)+k)/3k}\r\}.$$
Then $\widetilde{\boz_{\tilde h}}$ is a Sobolev $(p, q)$-extension domain by Theorem \ref{thm:BET}. Let $A:= E^{n-1}\times\lf(\frac{3}{2}, 2\r)$. Then 
$A\subset\partial\widetilde{\boz_{\tilde h}}$ and $|A|>0$. Let $0<r<\frac{1}{4}$ and $x\in A$ be arbitrary. 

We define a cut-off function on the ball $B(x, r)$ by setting
\begin{equation}
F_h(y)= \begin{cases} 1 &  \textnormal{in } B(x, \frac{r}{4}), \\
\frac{-4}{r}|y-x|+2 & \textnormal{in }B(x, \frac{r}{2})\setminus B(x,\frac{r}{4}),\\
 0 & \textnormal{in } B(x, r)\setminus B(x,\frac{r}{2}) \, . \end{cases}
\end{equation}

Set $u(y)=\chi_{\widetilde{\boz_{\tilde h}}}(y)$. Since $\widetilde{\boz_{\tilde h}}$ is a Sobolev $(p, q)$-extension domain, $E(u)\in W^{1,q}(\mathcal C_o)$. 
The function $v$ defined by $v(y):=F_h(y)E(u)(y)$ for $y\in B(x, r)$ satisfies 
$$v\in\mathcal W_q\lf(\widetilde{\boz_{\tilde h}}\cap B\lf(x, \frac{r}{4}\r), \widetilde{\boz_{\tilde h}}\cap A\lf(x; \frac{r}{2},\frac{3r}{4}\r); B(x,r)\r).$$ Pick $k_r\in\mathbb N$ so that $2^{-k_r-2}<r\leq 2^{-k_r-1}.$  
If $\mathsf C_k^j\cap B(x,r)\neq \emptyset,$ then, by the definition of $\mathsf C_k^j,$ we
have that $k>k_r.$ Moreover, by the definition of $E(u)$ in \eqref{eq:exten}, we have that
$$
|\nabla v(y)|\leq\begin{cases}\frac{C}{r_k}, & \ {\rm for\ every}\ y\in \mathsf C_k^{j}\cap B(x,r),\\
0, & \ {\rm elsewhere}\, .\end{cases}
$$
Hence, by the definition of $v$, we have
\begin{multline}
\int_{B(x, r)}|\nabla v(y)|^qdy\leq\sum_{k=k_r}^\fz\sum_{j}\int_{\mathsf C_k^{j}\cap B(x, r)}|\nabla v(y)|^qdy\nonumber\\
\leq C\sum_{k=k_r}^\fz r2^{(n-1)(k+1)}r_k^{n-1-q}\leq C\sum_{k=k_r}^\fz r2^{-k}\tilde h(8^{-k})\leq Crh(r).
\end{multline}
Thus
$$
Cap_q\lf(\widetilde{\boz_{\tilde h}}\cap B\lf(x, \frac{r}{4}\r), \widetilde{\boz_{\tilde h}}\cap A\lf(x; \frac{r}{2}, \frac{3r}{4}\r); B(x, r)\r)\leq Crh(r).
$$
This implies that
$$
\limsup_{r\to0^+}\frac{Cap_q\lf(\widetilde {\boz_{\tilde h}}\cap B\lf(x, \frac{r}{4}\r), \widetilde {\boz_{\tilde h}}\cap A\lf(x; \frac{r}{2},\frac{3r}{4}\r); B(x, r)\r)}{h(r)}\leq \lim_{r\to 0^+}Cr=0,
$$
as desired.
\end{proof}

\begin{rem}\label{rajatapaus} One can easily modify the construction of the domain from the previous proof so as to obtain a domain that fails to be $(n-1)$-fat at points of positive
volume of the boundary.  Let us sketch the necessary changes since we cannot use an extension operator as in the previous argument. 

First, define $r_k=2^{-k-2}\exp(-\exp(2^k))$ and $R_k=2^{-k-2}.$ Instead of $E(u)$ in the above computation, we use a function $u$ defined as follows. 
On each $Q_k^j\times [1,2),$ our function $u$ as a function of $(y',t)$
satisfies $u(y',t)=1$ if $y\in \mathsf D_k^j,$ $u(y',t)=0$ if $y'\notin B^{n-1}(x_k^j,R_k)$ and $u(y',t)=\frac {\log( \frac {R_k}{|y'-x_k^j|} ) } {\log( \frac {R_k}{r_k})}$ otherwise. 
Define $v(y)=F_h(y)u(y).$ Then a simple computation 
gives what we want.
\end{rem}

\section{Final comments}

In this section, we discuss in more detail some of the issues mentioned in the introduction and pose open problems that are motivated by the results in this paper.

First of all, let us comment on the locality of the estimate \eqref{equa:meade} from Theorem \ref{thm:point} that holds
for almost every $x$ for $0<r<r_x.$ When $q>n-1,$ we actually have this estimate for all $x$ and
all $0<r<\min\{1,\diam(\Omega)/4\}.$  This also holds when $q=1$ and $n=2.$

\begin{cor} \label{globaali}
Suppose that $1\le q<p$ when $n=2$ or that  $n-1<q<p$ when $n\ge 3.$ If $\boz$ is a Sobolev $(p,q)$-extension domain, then there is a nonnegative, bounded and countably additive set function $\Phi$ defined on  open sets, with the following property. For each $x\in\partial \boz$ and every $0<r<\min\{1,\frac 1 4 \diam(\boz)\},$ we have
\begin{equation}\label{vika}
\Phi(B(x,r))^{p-q}|B(x,r)\cap \boz|^q \ge |B(x,r)|^q.
\end{equation}
\end{cor}

This conclusion follows by combining Theorem \ref{thm:capa} with Remark \ref{vahva}, see inequality \eqref{vahvempi}.  Moreover, Theorem \ref{thm:capa} shows that \eqref{vika} holds also uniformly in
$x$ and $r$ for $1\le q\le n-1$ if we assume that \eqref{vahvempi} holds for these values. 

One can view the uniform validity of \eqref{vika} as the optimal analog of the Ahlfors-regularity condition \eqref{eq:regular}. In \cite{HKT, HKT1}, it was shown, relying on \eqref{eq:regular}, 
that a Sobolev $(p,p)$-extension domain can be equipped with a linear extension operator. We proved in Lemma \ref{lem:homo} that a Sobolev $(p,q)$-extension domain can be equipped with a homogeneous extension operator but we do not know if one could promote this to linearity. This motivates the following problem.

\begin{quest} Suppose that $\Omega$ is a bounded domain that satisfies the conclusion of Corollary
\ref{globaali}. Find the additional assumptions that ensure the existence of a linear extension operator from $W^{1,p}(\Omega)$ to $W^{1,q}(\rn).$
\end{quest} 

Given $1\le q<n-1,$ we constructed a Sobolev $(p,q)$-extension domain whose boundary has positive volume. We do not know if such domains exist also when $q=n-1>1.$

\begin{quest} Let $n\ge 3.$ Does there exist a Sobolev $(p,n-1)$-extension domain $\boz\subset \rn,$ for some $p>n-1,$ so that $|\partial \boz|>0$?
\end{quest}

Furthermore, our constructions of examples of $(p,q)$-extension domains with positive boundary volume have restrictions on 
$p$ in terms of $q.$ Even though these restrictions are natural for our constructions, we do not
know if some other constructions would allow $p$ to be arbitrarily close to $q.$

\begin{quest} Given $n\ge 3,$ $1\le q<n-1$ and $p>q,$ does there exist a Sobolev $(p,q)$-extension domain $\Omega\subset \rn$ whose boundary has positive volume?
\end{quest}

Finally, the reader familiar with \cite{HKT} may wonder why we do not employ the argument that was used there to prove \eqref{eq:regular} towards establishing \eqref{vika} in the case $q<n.$ We have indeed tried this
but without success.





\end{document}